\documentclass[a4paper,reqno,11pt]{amsart}
\usepackage{newtxtext}
\usepackage{microtype}
\usepackage{fix-cm}
\usepackage[T1]{fontenc}
\usepackage[numbers]{natbib}  
\usepackage{textcomp}
\usepackage{amsfonts} 
\usepackage{amssymb}
\usepackage{amsthm}
\usepackage{amsmath} 
\usepackage{mathrsfs} 
\usepackage{dsfont}
\usepackage{esint}
\usepackage{algorithm}
\usepackage{soul}
\usepackage{algpseudocode}
\usepackage[english]{babel} 
\usepackage{xcolor}
\definecolor{gray}{named}{gray}

\usepackage[left=2.8cm,right=2.8cm,top=3.5cm,bottom=3cm,headsep=0.8cm]{geometry}
\input{insbox}
\usepackage{tikz}
\usepackage{pgf,pgfplots,wrapfig}
\usetikzlibrary{arrows}
 
\usepackage[scriptsize,bf]{caption}

\usepackage{enumerate}
\usepackage{enumitem}

\usepackage{hyperref}
\usepackage{cleveref}
\usepackage{mathtools}
\hypersetup{
    colorlinks,
    linkcolor={red!80!black},
    citecolor={blue!80!black}
}

\theoremstyle{plain}
\begingroup
\theoremstyle{plain}
\endgroup

\newtheorem{theorem}{Theorem}[section]

\newtheorem{proposition}{Proposition}[section]

\newtheorem{lemma}{Lemma}[section]

\theoremstyle{definition}

\theoremstyle{remark}
\newtheorem{remark}{Remark}[section]

\theoremstyle{definition}
\theoremstyle{remark}

\numberwithin{equation}{section}

\makeatletter
\newcommand{\myitem}[1]{%
\item[#1]\protected@edef\@currentlabel{#1}%
}
\makeatother

\newcommand{\R}{\mathbb{R}}

\mathchardef\emptyset="001F

\renewcommand{\tilde}{\widetilde}

\author[Y-P. Choi]{Young-Pil Choi}
\address[Y-P. Choi]{Department of Mathematics, Yonsei University, Seoul 03722, Republic of Korea}
\email{ypchoi@yonsei.ac.kr}

\author[H-Z. Tang]{Houzhi Tang}
\address[H-Z. Tang]{School of Mathematics and Statistics, Anhui Normal University, Wuhu 241002, P.R. China}
	\email{houzhitang@ahnu.edu.cn}

\author[W-Y. Zou]{Weiyuan Zou}
\address[W-Y. Zou$^{1}$]{College of Mathematics and Physics, Beijing University of Chemical Technology, Beijing 100029, P.R. China}
\address[W-Y. Zou$^{2}$]{Mathematical Institute, University of Oxford, Oxford OX2 6GG, UK}
\email{zwy@amss.ac.cn}

\title[Global dynamics of damped Euler systems with exterior potentials]{Global dynamics of damped Euler systems with exterior potentials}

\keywords{Damped Euler equations, exterior potential, global well-posedness, optimal decay, blow-up analysis.}

\begin{document}
	
%
%
%
%
%
	\begin{abstract}
We study the three-dimensional isothermal Euler equations with linear damping and an exterior potential. For sufficiently large damping, we prove global well-posedness for arbitrarily large initial data by combining a parabolic comparison principle with scaled high-order energy estimates ensuring uniform density bounds. In the small-data regime with arbitrary damping, we establish global classical solutions and derive sharp algebraic decay rates via spectral analysis and frequency decomposition, and further prove their optimality under a mild non-degeneracy condition. Finally, for the pressureless damped system, we construct a weighted functional showing that solutions can blow up in finite time when the damping is insufficient, highlighting a qualitative difference from the pressured case.	
	\end{abstract}
	
	\maketitle

\tableofcontents

%
%
%
%
%
	\section{Introduction}

	In this paper, we are concerned with the global dynamics of the damped isothermal Euler equations with an exterior potential. The system reads
	\begin{equation}\label{Main1}
		\left\{
		\begin{aligned}	
			&\partial_t\rho+\textrm{div}(\rho u)=0,\\
			&\partial_t(\rho u)+\text{div}(\rho u\otimes u)+\nabla \rho = -\rho \nabla V-\gamma\rho u,
		\end{aligned}
		\right.
	\end{equation}
for $(x,t)\in\mathbb{R}^3\times\mathbb{R}_+$, where $\rho=\rho(x,t)>0$ and $u=u(x,t)$ denote the density and velocity fields, respectively. The parameter $\gamma>0$ represents the damping coefficient, and $V=V(x)\geq 0$ denotes a prescribed exterior potential. Throughout this paper, we assume that $V$ is sufficiently smooth and decays at infinity, namely,
	\begin{align*}
		\lim _{|x| \rightarrow +\infty}{V}(x,t)=0.
	\end{align*}
The initial data is given by
	\begin{align*}
		(\rho , u)(x,0)=(\rho_0,u_0),\quad \rho_0>0,
	\end{align*} 
and we impose the far-field condition
\begin{align*}
		\lim _{|x| \rightarrow +\infty}(\rho, u)(x,t)=\left(1, 0\right).
	\end{align*}
Note that $\rho_\infty=e^{-V}$ provides a stationary profile balancing the pressure and the potential forces. Since $V(x) \to 0$ as $|x| \to \infty$, we have $\rho_\infty(x) \to 1$, which is consistent with the far-field state. It is worth noting that the decaying nature of $V$ implies that the potential does not act as a confining force in the usual sense, in contrast to classical trapping potentials that diverge at infinity. Thus, the model describes a compressible isothermal flow subject to a localized external field and frictional damping, without a confining mechanism at large distances. The damping term $-\gamma \rho u$ introduces dissipation and plays a key role in stabilizing the dynamics, while the external potential locally perturbs the equilibrium state. This combination of localized forcing and damping makes the system a natural framework for investigating the relaxation toward the far-field equilibrium and understanding how the interplay of pressure, potential forces, and friction governs the global dynamics.

When $V=0$, the system \eqref{Main1} reduces to the classical damped Euler system with isothermal pressure, which has been extensively studied over the last decades. In the one-dimensional isentropic case, the global existence of BV and $L^\infty$ entropy solutions was established in \cite{DCP2009, DCL1989,HLY1998,HP2006}, and the long-time behavior of solutions was investigated in \cite{HL1992,HMP2005,HP2006,HPW2011,KWY2000}. For the multi-dimensional case, the global existence of small classical solutions and their large-time decay rates were obtained in \cite{BHN2008,HN2003,LWY2009,STW2003,TW2012}, where the analysis combined energy methods with Green's function techniques. The global well-posedness and decay in critical Besov spaces were further studied in \cite{FX2009,JZ2012,XZ2024}. In the isothermal case, the global existence of large solutions for sufficiently large damping coefficients has been proved in \cite{P2024}, together with the diffusive scaling limit showing convergence of the density to the solution of the heat equation as the relaxation time tends to zero. More recently, the global existence of large $L^2$ solutions for arbitrary damping was established in \cite{htzz2025}. When an exterior potential is present ($V\neq 0$), the damping interacts with the potential field to drive the solution toward the nontrivial stationary state $\rho_\infty$. Convergence to equilibrium in this setting has been studied using entropy methods \cite{kt2015}, and quantitative convergence in the 2-Wasserstein distance has been obtained under combined exterior and interaction potentials \cite{CCT2019}.

Despite these important advances, a rigorous understanding of the global well-posedness and large-time behavior for the damped isothermal Euler equations with a general exterior potential in multiple dimensions remains incomplete. In particular, while damping naturally induces dissipation, the presence of a spatially varying potential significantly alters the long-time dynamics, requiring refined energy estimates that capture the combined effects of pressure, damping, and confinement.   
 
Motivated by these observations, we now explain the main questions driving this work. The first motivation stems from the fact that most of the aforementioned results focus on the case of constant equilibrium states, while the analysis for non-constant equilibrium states remains largely unexplored. In particular, the recent work \cite{P2024} established the global existence of large solutions for the isothermal Euler system with damping under sufficiently large damping coefficients and constant equilibrium states. However, when an exterior potential is introduced, the equilibrium state becomes non-constant, and it is unclear whether such large solutions can still exist globally. This unresolved question serves as one of the primary motivations of the present study.

The second motivation is related to the decay rates of solutions. For small initial data, it is natural to expect global well-posedness for the damped Euler system with an exterior potential. However, the optimal decay rates toward non-constant equilibrium states remain unknown. In particular, the presence of an exterior potential significantly influences the dissipation mechanism, preventing the high-order derivatives from achieving the same decay rates as in the constant equilibrium case. This raises the question of how the exterior potential alters the long-time dynamics. Thus, we aim to investigate the essential differences in the large-time behavior between constant and non-constant equilibrium states.

The third motivation concerns the blow-up phenomena. It is well known that large solutions of the isentropic damped Euler system can develop singularities in finite time, as rigorously demonstrated in \cite{STW2003}. However, for the isothermal case, the blow-up problem remains highly challenging (see \cite{Chr19}). The primary difficulty lies in the absence of finite propagation speed, which plays a crucial role in the blow-up analysis for isentropic flows in \cite{STW2003}. To address this, we are inspired to investigate the pressureless damped Euler system as an intermediate step. Since our initial density is strictly away from vacuum, the classical blow-up criteria in \cite{Sider1985} are not directly applicable. Thus, we are motivated to explore whether blow-up analysis can still be carried out for this system under merely bounded damping coefficients. For recent developments on singularity formation in compressible Euler flows, we refer to \cite{BDSV23, BSV23, Chr19, IT24} and the references therein.

 \subsection{Notations}
We collect several notations and conventions used throughout the paper. 
\begin{itemize}
\item For $1\le p\le\infty$, $L^p(\mathbb{R}^3)$ denotes the usual Lebesgue space with norm $\|\cdot\|_{L^p}$, and $W^{k,p}(\mathbb{R}^3)$ the Sobolev space of order $k$, with norm $\|\cdot\|_{W^{k,p}}$. When $p=2$, we write $W^{k,2}(\mathbb{R}^3)=H^k(\mathbb{R}^3)$ with the norm $\|\cdot\|_{H^k}$. For simplicity, we often omit the domain, e.g.,
\[
	\|u\|_{H^k(\mathbb{R}^3)}=\|u\|_{H^k},\quad \|u\|_{L^p(\mathbb{R}^3)}=\|u\|_{L^p}.
\]
\item We denote by $C$ a generic positive constant, which may vary from line to line. 
\item We write $f_1\lesssim f_2$ if there exists a constant $C>0$ such that $f_1\leq C f_2$, and $f_1\sim f_2$ if $f_1 \lesssim f_2$ and $f_2 \lesssim f_1$.  
\item We use $f(x)=O(g(x))$ to indicate $\lim_{x\to0}\frac{f(x)}{g(x)}=C$ for some constant $C$.  
\item For an integer $k\ge0$, $\nabla^k f$ denotes the collection of all derivatives of order $k$, i.e.,
\[
\nabla^k f = \big\{ \partial_{x_1}^{\ell_1}\partial_{x_2}^{\ell_2}\partial_{x_3}^{\ell_3} f : |\ell|=\ell_1+\ell_2+\ell_3 = k \big\}.
\]
For a pair $(f,g)$ and a Banach space $X$, we write
\[
\|(f,g)\|_X := \|f\|_X + \|g\|_X.
\]
\item The Fourier transform of $f$ is denoted by $\hat{f}$ or $\mathcal{F}[f]$ and defined as
\[
		\hat{f}(\xi) =\mathcal{F}[f](\xi)=(2\pi)^{-\frac{3}{2}}\int_{\mathbb{R}^3} f(x)e^{-ix\cdot\xi}\,dx,\quad \xi\in\mathbb{R}^3.
\]
Its inverse is denoted by $\mathcal{F}^{-1}$. For $a\in\mathbb{R}$, the pseudodifferential operator $\Lambda^a$ is defined by
\[
\Lambda^a f := \mathcal{F}^{-1}\big(|\xi|^a \hat{f}(\xi)\big).
\]
\end{itemize}
	%
	%
	%
	%
	%
	%
	\subsection{Main results}\label{sec:main}
We now present the main results of this paper. 

Our first theorem concerns the large-damping regime, where global-in-time strong solutions exist for arbitrarily large initial data, provided the density stays strictly away from vacuum. This extends the existing theory for the damped Euler equations without potentials to the case of non-constant equilibrium states induced by an exterior potential.	
	
	Due to the isothermal pressure, it is convenient to introduce the logarithmic density $a=\ln\rho$, so that the system \eqref{Main1} takes the form
	\begin{equation}\label{Main2}
		\left\{
		\begin{aligned}	
			&\partial_ta+\text{div}u+u\cdot\nabla a=0,\\
			&\partial_tu+u\cdot\nabla u+\nabla a+\gamma u+\nabla V=0,
		\end{aligned}
		\right.
	\end{equation}
	with initial data and far-field condition:
	\begin{align*}
		(a,u)(x,0)=(a_0,u_0),\quad \text{and}\quad  \lim _{|x| \rightarrow+\infty}(a, u)(x,t)=\left(a_\infty, 0\right),
	\end{align*}
	where $a_0$ and $a_\infty$ are given by
	\[
		a_0=\ln \rho_0, \quad a_\infty=\ln\rho_\infty=-V.
		\]
	
	To study perturbations around the stationary state $(\rho_\infty,0)$, we define
	\[
	\phi=a-a_\infty,
	\]
	which leads to the reformulated system
	\begin{equation}\label{Main31}
		\left\{
		\begin{aligned}	
			&\partial_t\phi+\text{div}u=-u\cdot\nabla \phi-u\cdot\nabla a_\infty=-u\cdot\nabla \phi+u\cdot\nabla V:=f_1,\\
			&\partial_tu+\nabla \phi+\gamma u=-u\cdot\nabla u:=f_2,
		\end{aligned}
		\right.
	\end{equation}
	with initial data and far-field conditions
	\begin{align}\label{ID1}
		(\phi,u)(x,0)=(\phi_0,u_0),\quad \text{and}\quad  \lim _{|x| \rightarrow+\infty}(\phi, u)(x,t)=\left(0, 0\right).
	\end{align}		

	\begin{theorem}\label{thm1}
Let the initial data $(\phi_0, u_0)\in H^3(\mathbb{R}^3)$, and assume there exist constants $\rho_1,\rho_2>0$ such that
		\begin{align*}
			0<\rho_1\leq \rho_0(x)\leq \rho_2, \quad \forall \, x\in \mathbb{R}^3.	
		\end{align*}
		If the damping coefficient $\gamma$ is sufficiently large and the exterior potential satisfies $\|\nabla V\|_{H^3}\leq \epsilon_0$ with $\epsilon_0>0$ a small constant, then the system \eqref{Main31}-\eqref{ID1} admits a unique global regular solution $(\phi,u)$ satisfying 
			\begin{align*}
				&\|(\phi,u)(t)\|_{H^3}^2 
				+\int_0^t \left(\gamma\| u(s)\|_{H^3}^2 
				+\gamma^{-1}\|\nabla\phi(s)\|_{H^2}^2 \right) ds\leq C_0,
			\end{align*}
		where the positive constant $C_0=C(B_0)$ only depends on the initial data $B_0=\|(\phi_0,u_0)\|_{H^3}$.
	\end{theorem} 	
	
	\begin{remark}
Theorem \ref{thm1} demonstrates that strong damping effectively stabilizes the flow, ensuring global-in-time existence and uniform control of high-order Sobolev norms. In particular, if $V\equiv0$, Theorem \ref{thm1} recovers the result of \cite{P2024}, confirming consistency with the known theory. Moreover, for large $\gamma$, the momentum equation undergoes a diffusive relaxation, and \eqref{Main1} asymptotically reduces to the parabolic equation
		\begin{equation*}
			\left\{	
			\begin{aligned}
				&\partial_t\rho-\Delta \rho-{\rm{div}}(\rho \nabla V)=0,	\\
				&\rho(x,0)=\rho_0,
			\end{aligned}
			\right.
		\end{equation*}	
which reflects the dominant role of damping in driving the system toward equilibrium. The rigorous derivation of such parabolic limits can be obtained by following arguments in \cite{P2024}; see also \cite{CCT2019, CJ21, HPW2011, JR02, LT17} for related strong relaxation limits from the compressible Euler system to diffusive models such as the heat or porous medium equations.
	\end{remark}

Our second main result concerns the global well-posedness and sharp decay rates of solutions to the damped isothermal Euler system \eqref{Main1} for small initial data and arbitrary damping coefficients $\gamma>0$. Unlike the large-damping regime in Theorem \ref{thm1}, here we work in the small-data framework, allowing us to control the nonlinear terms without assuming any structural dominance of the damping.

We then establish the global existence of classical solutions for small initial perturbations, together with uniform-in-time energy bounds and explicit algebraic decay rates of low- and high-order derivatives.
	\begin{theorem}\label{thm3}
		Assume that $(\phi_0,u_0)\in H^3(\mathbb{R}^3)$ and $V\in H^4(\mathbb{R}^3)$. There exists a small positive constant $\delta_0$ such that if
		\begin{align*}
			\|(\phi_0,u_0)\|_{H^3}+\|\nabla V\|_{H^3}\leq \delta_0,
		\end{align*}
		then the Cauchy problem \eqref{Main31}-\eqref{ID1} admits a unique global-in-time regular solution 
		\[
		(\phi,u) \in L^\infty(\R_+; H^3(\R^3)) \times L^\infty(\R_+; H^3(\R^3))
		\]
satisfying the energy estimate
		\begin{equation}\label{main-est}
			\begin{aligned}
				\|(\phi, u)(t)\|_{H^3}^2
				+C \int_0^t\left(\| \nabla\phi(\tau)\|_{H^2}^2+\| u(\tau)\|_{H^3}^2\right) d\tau \leq C\delta_0^2,
			\end{aligned}
		\end{equation} 
		for any $t\in\mathbb{R}_+$. Moreover, if $(\phi_0,u_0) \in L^1(\R^3) \times L^1(\R^3)$, then the solution decays algebraically as
		\begin{align*}
			\|\nabla^k\phi(t)\|_{L^2}&\leq C(1+t)^{-\frac{3}{4}-\frac{k}{2}},\quad k=0,1,2,\nonumber\\
			\|\nabla^ku(t)\|_{L^2}&\leq  C(1+t)^{-\frac{5}{4}-\frac{k}{2}},\quad k=0,1,\nonumber\\
			\|\nabla^3\phi(t)\|_{L^2}&\leq C(1+t)^{-\frac{7}{4}},\nonumber\\
			\quad \|\nabla ^2u(t)\|_{H^1}&\leq C(1+t)^{-\frac{7}{4}},
		\end{align*}
		where the positive constant $C>0$ is independent of time and depends on $\delta_0$ and $\|(\phi_0,u_0)\|_{L^1}$.
	\end{theorem}
	\begin{remark}
		By using classical Sobolev interpolation inequalities, we obtain the $L^\infty$ decay rates of the solution as follows:
		\[
		\|\phi(t)\|_{L^\infty}\leq C(1+t)^{-\frac{3}{2}} \quad \mbox{and} \quad \|u(t)\|_{L^\infty}\leq C(1+t)^{-\frac{7}{4}},
		\]
		for some $C>0$ independent of $t$.
	\end{remark}
	
	\begin{remark}
		The presence of the coupling term $u\cdot\nabla V$ prevents further improvement of the high-order decay rates, particularly for $\|\nabla^3\phi\|_{L^2}$ and $\|\nabla^2 u\|_{L^2}$. 
	\end{remark}
	
Moreover, by imposing a mild non-degeneracy condition on the low-frequency part of the initial data, we show that these decay rates are in fact sharp, thereby fully characterizing the large-time dynamics of the system.
	\begin{theorem}\label{thm4}
		Assume that all the conditions of Theorem \ref{thm3} hold, and that the Fourier transform $\hat{\phi}_0$ of the initial density perturbation satisfies
		\begin{align}\label{in-data-optimal}
			\inf_{ |\xi|<r_0}|\hat{\phi}_0(\xi)|\geq c_0>0,
		\end{align}
for some $c_0>0$ and sufficiently small $r_0>0$. Then there exist constants $d_0,\bar d_0>0$, independent of $t$, such that for large $t$		
		\begin{align*}
			&d_0(1+t)^{-\frac{3}{4}-\frac{k}{2}}\leq \|\nabla^k\phi(t)\|_{L^2}\leq C(1+t)^{-\frac{3}{4}-\frac{k}{2}},\quad k=0,1,2, \\
			&\bar{d}_0(1+t)^{-\frac{5}{4}-\frac{k}{2}}\leq \|\nabla^ku(t)\|_{L^2}\leq C(1+t)^{-\frac{5}{4}-\frac{k}{2}},\quad k=0,1.
		\end{align*}
	\end{theorem}
\begin{remark}
The condition on the low-frequency part of $\hat{\phi}_0$ ensures that the diffusive modes dominate the long-time behavior, leading to sharp decay rates. This result confirms that the upper bounds in Theorem \ref{thm3} cannot, in general, be improved.
\end{remark}

We next revisit the pressureless Euler system with damping and no external potential:
	\begin{equation}\label{Main1111}
		\left\{
		\begin{aligned}	
			&\partial_t\rho+\textrm{div}(\rho u)=0,\\
			&\partial_t(\rho u)+\text{div}(\rho u\otimes u)+\gamma\rho u=0,
		\end{aligned}
		\right.
	\end{equation}
	with the initial data and the far-field condition:
	\begin{align}\label{BC111}
		(\rho,u)(x,0)=(\rho_0,u_0),\quad \text{and}\quad  \lim _{|x| \rightarrow+\infty}(\rho, u)(x,t)=\left(1, 0\right).
	\end{align}
	
Although damping introduces a stabilizing mechanism, it can be insufficient to prevent singularity formation when it is too weak. In contrast to the isothermal system studied in Theorems \ref{thm1}-\ref{thm4}, the absence of a pressure term eliminates the dispersive and smoothing effects that help maintain global regularity. Consequently, the flow retains the compressive features of the undamped Euler dynamics, and damping alone may fail to suppress the development of singularities.

We note that Riccati-type arguments have been employed in related blow-up analyses for compressible Euler flows. In particular, the work of \cite{CH09} treated the compressible Euler equations with pressure and showed that, under the assumptions of vanishing initial vorticity and compression vacuum states, the divergence of velocity along particle trajectories satisfies a scalar Riccati inequality, leading to finite-time blow-up. This idea was subsequently adapted in \cite[Appendix A]{HKK2014} for the pressureless Euler system with damping, where blow-up was obtained under the structural constraint
\[
\Omega_0(a)=\frac{1}{2}\left(\nabla u_0(a)-(\nabla u_0(a))^\perp\right), \quad \text{div}u_0(a)<-3\gamma.
\]
The zero-vorticity condition is essential in closing the Riccati inequality and thus significantly restricts the admissible class of initial data. In contrast, our approach relies on weighted energy-momentum functionals and does not impose this restriction, thereby yielding explicit blow-up criteria valid for a much broader range of initial configurations.

The following theorem establishes a precise blow-up criterion in this setting.
	\begin{theorem}\label{thm2}
Let $(\rho_0-1,u_0)\in H^3(\mathbb{R}^3)\times H^4(\mathbb{R}^3)$ with $\rho_0>0$ and  $\|u_0\|_{H^3}\leq a_0$ (not necessarily small). Define 
		\begin{align*}
			A_1(0):=\int_{\mathbb{R}^3}\rho_0 u_0\cdot\nabla H dx,\quad A_2(0):=\int_{\mathbb{R}^3}\rho_0 Hdx>0,\quad E_0:=\int_{\mathbb{R}^3}\rho_0|u_0|^2dx,
		\end{align*}
		where 
		\[
		H(x):=\gamma e^{-\frac{|x|^2}{\gamma}}.
		\]
If 
		\begin{align*}
			M_*:=\frac{-A_1(0)}{A_2(0)+E_0}>\left(8C_*^2a_0^2 \right)^{\frac{1}{5}}.
		\end{align*}
and the damping coefficient $\gamma$ satisfies	 
\[
			4	\left(\frac{C_*a_0}{M_*^2} \right)^2<\gamma <\frac{1}{2}M_*,
\]
		where $C_*$ is a positive constant from the Sobolev embedding inequality, then the strong solution to \eqref{Main1111}--\eqref{BC111} blows up in finite time.
	\end{theorem}

	\begin{remark}	
		In \cite{HKK2014}, blow-up was established for initial data belonging to the singular set
\[
			\mathcal{S}:=\big\{a \in \mathbb{R}^d:\Omega_0(a)=0,\quad\text{div}u_0(a)<-3\gamma\big\}.
\]
Our criterion in Theorem \ref{thm2} goes beyond this structural restriction. To illustrate, consider small damping $\gamma>0$ and the initial data		
		\begin{align*}
			\rho_0=1,\quad u_0=(x_1,0,x_3)e^{-\frac{|x|^2}{\gamma}},\quad a_0=\gamma^\frac{5}{4}.
		\end{align*}
		A direct calculation yields
	\begin{align*}
			\|u_0\|_{L^2}\sim \gamma^{\frac{5}{4}},\quad \|\nabla u_0\|_{L^2}\sim \gamma^{\frac{3}{4}},
		\end{align*}
		and 
		\begin{align*}
			A_1(0) \sim -\gamma^{\frac{5}{2}}, \quad A_2(0) \sim \gamma^{\frac{5}{2}}, \quad E_0 \sim \gamma^{\frac{5}{2}}.
		\end{align*}
		This implies
		\begin{align*}
			M_*=\frac{-A_1(0)}{A_2(0)+E_0}=O(1)\geq \left(8C_*^2a_0^2 \right)^{\frac{1}{5}}=\left(8C_*^2\right)^{\frac{1}{5}}\gamma^{\frac{1}{2}},
		\end{align*}
		and 
		\begin{align*}
				4	\left(\frac{C_*a_0}{M_*^2} \right)^2=4	\left(\frac{C_*}{M_*^2} \right)^2\gamma^{\frac{5}{2}}<\gamma <\frac{1}{2}M_*.
		\end{align*}

Thus, the conditions of Theorem \ref{thm2} are satisfied. To further examine the vorticity, let us compute the gradient of the initial velocity:
		\begin{align*}
		\nabla u_0(x) = e^{-\frac{|x|^2}{\gamma}} \begin{pmatrix} 
			1 - \frac{2x_1^2}{\gamma} & -\frac{2x_1 x_2}{\gamma} & -\frac{2x_1 x_3}{\gamma} \\
			0 & 0 & 0 \\
			-\frac{2x_1 x_3}{\gamma} & -\frac{2x_2 x_3}{\gamma} & 1 - \frac{2x_3^2}{\gamma}
		\end{pmatrix}.
		\end{align*}	
For instance, if we choose $x=a=(a_1,0,a_3)\neq 0$, the associated vorticity tensor is given by	
		\begin{align*}
			\Omega_0(a)=\frac{1}{2}(\nabla u_0(a)-(\nabla u_0(a))^\perp)\neq 0,
		\end{align*}
which is non-vanishing.
	
This example shows that blow-up can occur even without the zero-vorticity condition from \cite{HKK2014}. In particular, small initial data may still lead to singularity formation when damping is weak, highlighting the critical role of damping in stabilizing the flow.
	\end{remark}

\begin{remark}
Due to the isothermal pressure, the density loses its finite propagation property. Therefore, classical blow-up strategies for isentropic Euler or pure Euler systems \cite{Sider1985,STW2003} fail in this setting, making the blow-up analysis for the isothermal damping system significantly more challenging.
\end{remark}

	%
	%
	%
	%
	%
	%
	
\subsection{Strategy of proof}
We now outline the key ideas employed in the proofs of the main theorems.

For Theorem \ref{thm1}, inspired by the work \cite{P2024} for the damped Euler system without external potentials, we aim to establish global well-posedness for the full system with an exterior potential, where the equilibrium state is no longer constant. The main challenge is that the momentum formulation plays a crucial role in \cite{P2024}, while in our case the natural formulation involves the velocity. To overcome this, we introduce the \emph{a priori} assumption 
\begin{align*}
&\gamma^{-\frac{1}{2}}\|(\rho-\rho_\infty, m, u)(t)\|_{H^3} +\left(\int_0^t\|(m,u)(s)\|_{H^3}^2 \,ds\right)^{\frac{1}{2}} \leq \delta,\\
	& \int_0^t\|\nabla(\rho-\rho_\infty)(s)\|_{H^2}^2 ds \leq \gamma \mathcal{R}_0^2,
\end{align*}
for some positive constants $\delta > 0$ and $\mathcal{R}_0$, and construct energy estimates that exploit the smallness of $\gamma^{-1}$ and $\delta$. A central step is to guarantee uniform-in-time lower and upper bounds for the density. For this, we employ a parabolic comparison argument: we construct an auxiliary function $\rho_*$ solving the damped diffusion equation
\[
\partial_t\rho_*-\gamma^{-1}\Delta\rho_*-\gamma^{-1}{\rm{div}}(\rho_*\nabla V)=0,
\]
and use a maximum principle to prove that $\rho_*$ remains strictly between two positive constants (Lemma \ref{lemM}). Comparing the true solution $\rho$ with $\rho_*$ then allows us to propagate these bounds to the Euler system (Proposition \ref{A2}). With this uniform control, we derive high-order energy estimates in Sobolev spaces and close the bootstrap bound. Finally, by a continuity argument, we extend the local solution to a unique global classical one, establishing Theorem \ref{thm1}.

For Theorems \ref{thm3} and \ref{thm4}, we study the small-data regime with arbitrary damping coefficients. We first prove global well-posedness via standard energy methods applied to the reformulated perturbation system.  To investigate the large-time behavior, we perform a spectral analysis of the linearized problem: the Green's function is decomposed into low- and high-frequency components, and Duhamel's principle is used to control the nonlinear terms. This approach yields the algebraic decay rates
\[ 
\|\phi(t)\|_{H^3} \lesssim (1+t)^{-\frac{3}{4}}, \quad \|u(t)\|_{H^3}\lesssim (1+t)^{-\frac{5}{4}}.
\] 
To improve the decay of higher derivatives, we further refine the analysis by introducing a carefully constructed energy functional that isolates the dissipative contribution of the low-frequency modes. This refinement leads to improved decay estimates for $\|\nabla^2 (\phi, u)(t)\|_{H^1}$, although the presence of the potential limits the attainable rates for the highest-order derivatives. Finally, under the mild non-degeneracy condition \eqref{in-data-optimal} on the low-frequency part of the initial data, we derive matching lower bounds for $\|\nabla^k (\phi, u)(t)\|_{L^2}$, thereby establishing the optimality of the obtained decay rates.

For Theorem \ref{thm2}, we turn to the pressureless Euler system with damping. Here, the absence of pressure destroys the dispersive smoothing effects, making the flow prone to compression-driven singularities. Since the initial density is strictly positive, classical vacuum-based blow-up arguments (e.g. \cite{Sider1985}) do not apply. Instead, we introduce the weighted momentum functionals
\[
A_1(t):=\int_{\mathbb{R}^3}\rho u\cdot\nabla H dx,\quad A_2(t):=\int_{\mathbb{R}^3}\rho Hdx>0
\]
with a Gaussian weight 
		\[
		H(x)=\gamma e^{-\frac{|x|^2}{\gamma}}.
		\]
This choice of $H$ is critical: its Gaussian profile ensures fast spatial decay, eliminating boundary contributions, and its $\gamma$-dependent scaling balances the nonlinear transport and damping terms. Differentiating $A_2(t)$ twice and combining with the momentum equation yields that $A_2(t)$ satisfies a \emph{second-order differential inequality} of the form
	\begin{align*}
		\frac{d^2A_2(t)}{dt^2}+\gamma \frac{dA_2(t)}{dt}\leq D_*A_2(t)+D_*E_0,
	\end{align*}
where $D_*$ depends on the initial data and damping. Careful analysis of the roots of the corresponding characteristic equation shows that, under suitable assumptions on the initial data and damping coefficient, the solution $A_2(t)$ becomes negative in finite time, contradicting the positivity of $A_2$. This contradiction implies that the classical solution must blow up in finite time.

	%
	%
	%
	%
	%
	%
\subsection{Outline of the paper}
 The remainder of the paper is organized as follows. In Section \ref{sec.2}, we provide the local-in-time well-posedness of the reformulated system and collect auxiliary tools, including commutator bounds, interpolation inequalities, and frequency decomposition estimates. Section \ref{Sec:thm1} addresses the large-damping regime and proves global well-posedness without smallness assumptions (Theorem \ref{thm1}). In Section \ref{Sec:thm3-4}, we turn to small perturbations with arbitrary damping, deriving uniform energy bounds and explicit decay rates via a time-weighted energy method and frequency decomposition (Theorem \ref{thm3}). Section \ref{sec:opt-decay} establishes sharp lower bounds for these decay rates under a mild non-degeneracy assumption, confirming their optimality (Theorem \ref{thm4}). Finally, Section \ref{Sec:thm2} analyzes the pressureless damped Euler system, constructing weighted energy-momentum functionals to prove finite-time blow-up for insufficient damping (Theorem \ref{thm2}).

	%
	%
	%
	%
	%
	%
\section{Preliminaries}\label{sec.2}	
In this section, we collect several auxiliary results and technical tools that will be used throughout the paper. We first establish the local-in-time well-posedness of the reformulated system \eqref{Main31}, which provides the foundation for the subsequent global existence analysis. Then we present a collection of functional inequalities, commutator estimates, frequency decompositions, and interpolation lemmas that will be repeatedly invoked in our energy estimates and decay analysis.
	%
	%
	%
	%
	%
	%
\subsection{Local-in-time well-posedness}  As a first step, we recall that the Cauchy problem for the reformulated system \eqref{Main31} admits a unique local classical solution for sufficiently regular initial data. The proof follows the standard approach based on the contraction mapping principle combined with Sobolev embeddings (see also \cite{STW2003} for a related argument). For completeness, we state the result without proof.
	\begin{theorem}\label{lem-loc}
		Assume that the initial data satisfy $(\phi_0,u_0)\in H^3(\mathbb{R}^3) \times H^3(\mathbb{R}^3)$. Then there exists a short time $T_0>0$ such that the system \eqref{Main31} admits a unique classical solution $(\phi,u)$ satisfying
		\begin{align*}
			&\phi \in C^0([0,T_0];H^3(\mathbb{R}^3))\cap C^1([0,T_0];H^2(\mathbb{R}^3)),\\
			&u\in C^0([0,T_0];H^3(\mathbb{R}^3))\cap C^1([0,T_0];H^2(\mathbb{R}^3)).
		\end{align*}
	\end{theorem}
	This local result allows us to carry out \emph{a priori} estimates on the solution and extend the existence interval to arbitrary times under smallness or structural conditions.
	%
	%
	%
	%
	%
	%
\subsection{Auxiliary lemmas} We now introduce several technical lemmas that will play a crucial role in our subsequent analysis.

To handle the nonlinear terms in the energy estimates, especially those involving products of high-order derivatives, we use the following well-known commutator bound.
	\begin{lemma}{\rm (\cite[Lemma A.3]{MR2917409})}\label{lema2}
		Let $m\geq 1$ be an integer and define the communicator 
		\begin{equation*}
			[\nabla^m,f]g=\nabla^{m}(fg)-f\nabla^m g.
		\end{equation*}
		Then we have
		\begin{equation*}
			\big\|[\nabla^m,f]g\big\|_{L^p}\lesssim \|\nabla f\|_{L^{p_1}}\|\nabla^{m-1}g\|_{L^{p_2}}+
			\|\nabla^m f\|_{L^{p_3}}\|g\|_{L^{p_4}},
		\end{equation*}
for $p, p_2, p_3\in(1,+\infty)$ with
		\begin{equation*}
			\frac{1}{p}=\frac{1}{p_1}+\frac{1}{p_2}=\frac{1}{p_3}+\frac{1}{p_4}.
		\end{equation*}
	\end{lemma}
	
In the analysis of finite-time breakdown of smoothness of solutions based on energy functionals, we encounter second-order differential inequalities. The next lemma provides an explicit bound for such cases.	
	
	\begin{lemma}{\rm (\cite[Lemma 3.7]{CJ25})}\label{lem2}
		Let $h = h(t)$ be a nonnegative $\mathcal{C}^2$-function satisfying  
		\begin{align*}
			h''(t)+c_1h'(t)\leq c_2h(t)+c_3,\quad h(0)=h_0,\quad h'(0)=h_0'
		\end{align*}
		for some $c_i > 0, i = 1,2$ and $c_3\in\mathbb{R}$. Then we have
		\begin{align*}
			h(t)&\leq\left(h_0+\frac{c_3}{\beta(\beta + c_1)}+\frac{1}{c_1 + 2\beta}\left(h_0'-\beta h_0-\frac{c_3}{\beta + c_1}\right)\right)e^{\beta t}\\
			&\quad -\frac{1}{c_1 + 2\beta}\left(h_0'-\beta h_0-\frac{c_3}{\beta + c_1}\right)e^{-(c_1+\beta)t}-\frac{c_3}{\beta(c_1 + \beta)},
		\end{align*}
		where $\beta>0$ is given by
		\begin{align*}
			\beta:=\frac{-c_1+\sqrt{c_1^2 + 4c_2}}{2}.
		\end{align*}
	\end{lemma}
	
To interpolate between low- and high-order norms, we will use the following fractional Gagliardo--Nirenberg type estimate.	
	
		\begin{lemma}\label{lema5}
		Let $a\geq 0$ and integer $l \geq 0$, then we have
		\begin{equation*}
			\|\nabla^l f\|_{L^2}\lesssim \|\nabla^{l+1}f\|_{L^2}^{1-\theta}\|\Lambda^{-a}f\|_{L^2}^\theta\quad \text{with}\quad \theta=\frac{1}{1+l+a}.
		\end{equation*}
	\end{lemma}
	\begin{proof}		
		By Parseval's identity and H\"{o}lder's inequality, we easily find
		\begin{equation*}
			\|\nabla^l f\|_{L^2}=\Big\|(i\xi)^l\hat{f}\Big\|_{L^2}
			\lesssim\big\|(i\xi)^{l+1}\hat{f}\big\|_{L^2}^{1-\theta}\big\||\xi|^{-a}\hat{f}\big\|_{L^2}^\theta=\|\nabla^{l+1}f\|_{L^2}^{1-\theta}\|\Lambda^{-a}f\|_{L^2}^\theta,
		\end{equation*}
		for $\theta=\frac{1}{1+l+a}$. This completes the proof.
	\end{proof}
	
For the decay analysis, we decompose functions into low- and high-frequency components. Define the operators $\mathcal{K}_{1}$ and $\mathcal{K}_{\infty}$ on $L^{2}$ as
	\begin{equation}\label{THZ5}
		\mathcal{K}_{1}f=f^\ell :=\mathcal{F}^{- 1}\big(\hat\chi_{1}(\xi)\mathcal{F}[f](\xi)\big) \quad \mbox{and} \quad	\mathcal{K}_{\infty}f=f^h :=\mathcal{F}^{- 1}\big(\hat\chi_{\infty}(\xi)\mathcal{F}[f](\xi)\big),
	\end{equation}
	respectively, where $\hat{\chi}_{j}(\xi)(j=1,\infty)\in C^{\infty}(\mathbb{R}^{3})$, $0\leq \hat\chi_{j}\leq 1$ are smooth cut-off functions satisfying
	\begin{equation*}
		\hat\chi_{1}(\xi)=\left\{
		\begin{array}{l}
			1\quad (|\xi|\leq r_{0})\\
			0\quad (|\xi|\geq R_0)
		\end{array}
		\right.
		\quad \mbox{and} \quad \hat\chi_{\infty}(\xi)=1-\hat\chi_{1}(\xi),
	\end{equation*}
with sufficiently small $r_0>0$ and large $R_0>0$. This decomposition allows us to separate the diffusive low-frequency modes from the dissipative high-frequency modes.
	\begin{lemma}{\rm (\cite[Lemma 2.4]{huang2024global})}\label{lema6}
		Let $m > n \geq 0$. Then for $f\in H^m$, there exists a constant $C > 0$ such that 
		\begin{equation*}
			\|\nabla^{m} f^\ell\|_{L^2}\leq C\|\nabla^{n} f^\ell\|_{L^2},\quad \|\nabla^{n}f^h\|_{L^2}\leq C\|\nabla^{m}f^h\|_{L^2},
		\end{equation*}
		and 
		\begin{equation*}
			\|\nabla^{n} f^\ell\|_{L^2}\leq C\|\nabla^{n} f\|_{L^2},\quad\|\nabla^{n}f^h\|_{L^2}\leq C\|\nabla^{n}f\|_{L^2}.
		\end{equation*}
	\end{lemma}
	This frequency splitting, together with the above interpolation and commutator estimates, forms the core analytic framework for deriving uniform energy bounds and sharp decay rates in later sections.


\section{Global well-posedness with large damping for arbitrarily large initial data}\label{Sec:thm1}

In this section, we establish the global-in-time well-posedness of strong solutions to the damped isothermal Euler system with an exterior potential in the regime of large damping coefficients. A notable feature of our result, as stated in Theorem \ref{thm1}, is that {\it no smallness assumption on the initial data is required}: strong damping effectively suppresses nonlinear instabilities and ensures global regularity even for large initial perturbations. Our approach quantifies this stabilization mechanism by constructing an appropriately scaled energy functional that captures the underlying dissipative structure of the system and closes a bootstrap argument.

We begin by rewriting the Euler system with damping in the momentum formulation:
\begin{equation}\label{MF1}
	\left\{
	\begin{aligned}	
		&\partial_t\rho+\textrm{div}m=0, \quad m= \rho u,\\
		&\partial_tm+\nabla \rho+\gamma m+\rho\nabla V=-\text{div}\Big(\frac{m\otimes m}{\rho}\Big),
	\end{aligned}
	\right.
\end{equation}
with initial data 
\begin{align*}
	(\rho, m)(x,0)=(\rho_0,m_0), \quad m_0=\rho_0u_0.	
\end{align*}

To control the solution uniformly in time, we introduce the energy functional
\begin{align*}
	Z_1(t)&:=\gamma^{-\frac{1}{2}}\|(\rho-\rho_\infty, m, u)(t)\|_{H^3} +\left(\int_0^t\|(m,u)(s)\|_{H^3}^2 \,ds\right)^{\frac{1}{2}},\\
	Z_2(t)&:=\left(\gamma^{-1}\int_0^t \|\nabla(\rho-\rho_\infty)(s)\|_{H^2}^2\,ds\right)^{\frac{1}{2}},
\end{align*}
which combines the instantaneous Sobolev norms with time-integrated dissipation. The prefactor $\gamma^{-1/2}$ reflects the scaling balance between the large damping term and the convective nonlinearities, ensuring that the energy remains small for large $\gamma$.



\subsection{Reformulation and bootstrap setting}
To carry out the global existence proof, we assume the following \emph{a priori} bound for some small constant $\delta>0$, a bounded constant $\mathcal{R}_0$ (to be determined the later) such that 
\begin{equation}\label{a-priori}
	Z_1(t)\leq \delta, \quad 	Z_2(t)\leq \mathcal{R}_0.
\end{equation}
This assumption implies the estimates
\begin{align}\label{030607}
\begin{aligned}
&\|(\rho-\rho_\infty,m,u)(t)\|_{H^3}\leq \gamma^{\frac{1}{2}}\delta, \quad \int_0^t\|(m,u)(s)\|_{H^3}^2ds\leq \delta^2, \cr
&\qquad  \int_0^t\|\nabla(\rho-\rho_\infty)(s)\|_{H^2}^2ds\leq \gamma \mathcal{R}_0^2,
\end{aligned}
\end{align}
which will be repeatedly used to control nonlinear contributions. The bootstrap argument then proceeds by improving this bound and closing it for sufficiently small $\delta$ and large $\gamma>0$, thus extending the local solution globally.

For later use, we also introduce the shorthand notation for the initial data size:
\[ 
 B_0:=\|(\phi_0, u_0)\|_{H^3}. 
\] 
 Using the definition of momentum and the smallness of $\epsilon_0$, we have
\begin{align}\label{030602}
\begin{aligned}
	\|m_0\|_{H^3}&\leq \|(\rho_0-\rho_\infty)u_0\|_{H^3}+\|\rho_\infty u_0\|_{H^3}\\
	&\leq C\|\rho_0-\rho_\infty\|_{H^3}\|u_0\|_{H^3}+C\|\rho_\infty\|_{L^\infty}\|u_0\|_{H^3}+\|\nabla\rho_{\infty}\|_{H^2}\|u_0\|_{H^3}\\
	&\leq C\|(e^{\phi_0}-1)\rho_\infty\|_{H^3}\|u_0\|_{H^3}+C\|u_0\|_{H^3}+C\|\nabla V\|_{H^3}\|u_0\|_{H^3}\\
		&\leq C\|\phi_0\|_{H^3}\|\rho_\infty\|_{L^\infty}\|u_0\|_{H^3}+C\|\phi_0\|_{H^3}\|\nabla\rho_\infty\|_{H^2}\|u_0\|_{H^3}\cr
		&\quad +C\|u_0\|_{H^3}+C\|\nabla V\|_{H^3}\|u_0\|_{H^3}\\
		&\leq C\|\phi_0\|_{H^3}\|u_0\|_{H^3}+C\|\phi_0\|_{H^3}\|\nabla V\|_{H^3}\|u_0\|_{H^3}\cr
		&\quad +C\|u_0\|_{H^3}+C\|\nabla V\|_{H^3}\|u_0\|_{H^3}\\
	&\leq C(B_0^2+\epsilon_0 B_0^2+B_0+\epsilon_0B_0)\\
	&\leq C(B_0^2+B_0),
	\end{aligned}
\end{align}
Note that $\rho_\infty = e^{-V}$ is not in $L^2(\R^3)$, so the above bounds rely on its $L^\infty$-control and Sobolev regularity of $V$. This shows that the initial momentum is controlled by the size of the density and velocity perturbations.

Rewriting the system for the perturbation variables $(\rho-\rho_\infty,m)$ yields
\begin{equation}\label{MF}
	\left\{
	\begin{aligned}	
		&\partial_t(\rho-\rho_\infty)+\textrm{div}m=0,\\
		&\partial_tm+\nabla (\rho-\rho_\infty) +\gamma m+(\rho-\rho_\infty) \nabla V=-\text{div}\Big(\frac{m\otimes m}{\rho}\Big),
	\end{aligned}
	\right.
\end{equation}
with initial data 
\[
	(\rho-\rho_\infty,m)(x,0)=(\rho_0-\rho_\infty,m_0),\quad m_0=\rho_0u_0.	
\]

%
%
%
%
%
%
%
%
%
%
%
\subsection{Uniform density bounds}
An essential step in the global analysis is to ensure that the density remains uniformly bounded away from vacuum and infinity. This prevents degeneracy in the nonlinear terms involving $1/\rho$ and allows the Sobolev estimates for $m/\rho$ to remain valid. To this end, we first study the parabolic problem satisfied by a comparison function $\rho_*$:
\begin{lemma}\label{lemM}
 Let $\rho_{*}(x,t)>0$ be a solution to the following equation:
	\begin{equation}\label{031401}
		\left\{
		\begin{aligned}
			&\partial_t\rho_*-\gamma^{-1}\Delta\rho_*-\gamma^{-1}{\rm{div}}(\rho_*\nabla V)=0,\\
			& \rho_*(x,0)=\rho_0,
		\end{aligned}
		\right.
	\end{equation}
for $0<t\leq T$, with $\rho_0-\rho_\infty \in H^3(\mathbb{R}^3)$, $\|\nabla V\|_{H^3}\leq \epsilon_0$, and $0<\rho_\infty=e^{-V}\leq 1$. If the damping coefficient $\gamma$ is sufficiently large, $\epsilon_0$ is small enough,  and the initial density $\rho_0$ satisfies $\rho_1\leq \rho_0\leq \rho_2$  for some positive constants $\rho_1,\rho_2$, then for all $0<t\le T$,
	\begin{align*}
		\rho_1\leq \rho_*\leq \rho_2.	
	\end{align*}
\end{lemma}
\begin{remark}
This comparison lemma ensures that the parabolic relaxation of the density preserves its initial lower and upper bounds for large damping. In particular, it guarantees that $\rho$ remains strictly positive, which is essential for controlling the nonlinear fluxes in \eqref{MF}.
\end{remark}
\begin{proof}[Proof of Lemma \ref{lemM}]
The proof combines an $H^3$ energy estimate for the parabolic problem \eqref{031401} with a maximum principle argument. We divide the proof into two parts: the upper bound and the lower bound for $\rho_*$.

\medskip
\noindent \emph{Step 1: Uniform $H^3$ control and preliminary bound.} Defining $q:=\rho_*-\varepsilon t$, we find from \eqref{031401} that $q$ satisfies
\[
		\partial_tq-\gamma^{-1}\Delta q-\gamma^{-1}\nabla q\cdot\nabla V=\gamma^{-1}\rho_*\Delta V-\varepsilon.	
\]
Noticing $\rho_\infty=e^{-V}$, we obtain
	\begin{equation}\label{H1}
		\left\{
		\begin{aligned}
			&\partial_t(\rho_*-\rho_\infty)-\gamma^{-1}\Delta(\rho_*-\rho_\infty)-\gamma^{-1}\text{div}((\rho_*-\rho_\infty)\nabla V)=0,\\
			& (\rho_*-\rho_\infty)(x,0)=\rho_0-\rho_\infty.	
		\end{aligned}
		\right.
	\end{equation}
Then a direct computation yields
	\begin{align*}
		&\frac{1}{2}\frac{d}{dt}\|\rho_*-\rho_\infty\|_{H^3}^2+\gamma^{-1}\|\nabla(\rho_*-\rho_\infty)\|_{H^3}^2\\
		&\quad =\gamma^{-1}\sum_{k=0}^3\int_{\mathbb{R}^3}\nabla^k\text{div}((\rho_*-\rho_\infty)\nabla V)\cdot \nabla^k(\rho_*-\rho_\infty)dx\\
		&\quad \leq C\gamma^{-1}\|\nabla V\|_{H^3}\|\nabla(\rho_*-\rho_\infty)\|_{H^3}^2\\
		&\quad \leq C\epsilon_0\gamma^{-1}\|\nabla(\rho_*-\rho_\infty)\|_{H^3}^2.
	\end{align*}
		Due to the smallness of $\epsilon_0$,  the last term can be absorbed by the dissipation, giving
	\begin{align*}
		\frac{d}{dt}\|\rho_*-\rho_\infty\|_{H^3}^2+\gamma^{-1}\|\nabla(\rho_*-\rho_\infty)\|_{H^3}^2\leq 0,
	\end{align*}
	which implies that 
	\begin{align*}
		&\|(\rho_*-\rho_\infty)(t)\|_{H^3}^2+\gamma^{-1}\int_0^t\|\nabla(\rho_*-\rho_\infty)(s)\|_{H^3}^2ds\\
		&\quad \leq C\|\rho_0-\rho_\infty\|_{H^3}^2\\
		&\quad \leq C \|(e^{\phi_0}-1)\rho_\infty\|_{H^3}\\
		&\quad\leq C\|\phi_0\|_{H^3}\|\rho_\infty\|_{L^\infty}+C\|\phi_0\|_{H^3}\|\nabla \rho_\infty\|_{H^2}\\
		&\quad\leq C(B_0+\epsilon_0 B_0)\\
		&\quad\leq CB_0.
\end{align*}
	By Sobolev embedding, we prove that there exists a positive constant $\tilde{C}_0$ such that 
		\begin{align*}
			\rho_*\leq \|\rho_*-\rho_\infty\|_{L^\infty}+\rho_\infty \leq  C\|\rho_*-\rho_\infty\|_{H^3}+\rho_\infty\leq CB_0+\rho_\infty\leq \tilde{C}_0.
		\end{align*}

\medskip
\noindent
\emph{Step 2: Upper bound via the maximum principle.} Due to the fact that $\gamma$ is sufficiently large and $\|\nabla V\|_{H^3}\leq \epsilon_0$, we deduce
	\begin{align*}
		\partial_tq-\gamma^{-1}\Delta q-\gamma^{-1}\nabla q\cdot\nabla V
		\leq C\gamma^{-1}\rho_*\|\nabla V\|_{H^3}-\varepsilon\leq C\gamma^{-1}\tilde{C}_0\epsilon_0-\varepsilon\leq -\frac{\varepsilon}{2}<0.
	\end{align*}
If $(x_0,t_0)$ is an interior maximum of $q$ in $\Omega=\{(x,t) : x\in \mathbb{R}^3, 0<t\leq T\}$, then at that point  
	\begin{align*}
		\partial_tq(x_0,t_0)\geq 0, \quad \Delta q(x_0,t_0)\leq 0, \quad \nabla q(x_0,t_0)=0,
	\end{align*}
	which implies that 
	\begin{align*}
		\partial_tq-\gamma^{-1}\Delta q-\gamma^{-1}\nabla q\cdot\nabla V\geq 0.
	\end{align*}
	This contradicts our assumption that $(x_0,t_0)\in \Omega$. Thus the maximum is achieved on the parabolic boundary, implying
	\begin{align*}
		q(x,t)=\rho_*-\varepsilon t\leq \max\{q_0\}\leq\max{\rho_0}\leq \rho_2.
	\end{align*}
Hence, we have
	\begin{align*}
		\rho_*=q+\varepsilon t\leq \rho_2+\varepsilon t\leq \rho_2+\varepsilon T.
	\end{align*}
Letting $\varepsilon\to0$, we obtain $\rho_*\le \rho_2$.

\medskip
\noindent
\emph{Step 3: Lower bound via the maximum principle.} Similarly, define $q_*=\rho_*+\varepsilon t$, which satisfies
\[
		\partial_tq_*-\gamma^{-1}\Delta q_*-\gamma^{-1}\nabla q_*\cdot\nabla V=\gamma^{-1}\rho_*\Delta V+\varepsilon>-C\gamma^{-1}\tilde{C}_0\epsilon_0+\varepsilon> \frac{\varepsilon}{2}>0.
\]
If $(x_0,t_0)$ is an interior minimum of $q_*$, then at that point $\partial_tq_* \le 0$, $\Delta q_* \ge0$, and $\nabla q_*=0$, again leading to a contradiction. Thus the minimum occurs on the boundary, yielding
	\begin{align*}
		q_*(x,t)=\rho_*+\varepsilon t\geq \min\{q_0\}\geq\min{\rho_0}\geq \rho_1,
	\end{align*}
and hence
	\begin{align*}
		\rho_*=q_*-\varepsilon t\geq \rho_1-\varepsilon t\geq \rho_1-\varepsilon T.
	\end{align*}
Letting $\varepsilon\to0$, we obtain $\rho_*\ge\rho_1$.  

\medskip

Combining the two bounds, we conclude that
	\begin{align*}
		\rho_1\leq \rho_*(x,t)\leq \rho_2, \quad \forall\, x \in \R^3, \ 0 < t \leq T,	
	\end{align*}
which completes the proof.
\end{proof}

Having established in Lemma \ref{lemM} that the parabolic comparison function $\rho_*(x,t)$ remains uniformly bounded between two positive constants, our next task is to transfer this property to the actual solution $\rho(x,t)$ of the nonlinear Euler system.  To achieve this, we introduce a carefully chosen auxiliary variable that couples the density and momentum, compare $\rho$ with the parabolic proxy $\rho_*$, and exploit the strong damping to derive a uniform bound.

\begin{proposition}\label{A2}
	Assume that all the conditions of Theorem \ref{thm1} hold. Then for all $(x,t)\in\mathbb{R}^{3}\times [0,T)$, we have
	\begin{align*}
		\frac{1}{2}\rho_1\leq \rho(x,t)\leq \frac{3}{2}\rho_2.	
	\end{align*}
\end{proposition}
\begin{proof} We proceed by introducing the auxiliary variable $w$ that couples $\rho$ and $m$ and by comparing $\rho$ to the parabolic relaxation $\rho_*$ from Lemma \ref{lemM}. This allows us to propagate the pointwise bounds of $\rho$ uniformly in time.

\medskip
\noindent \emph{Step 1: Introduction of $w$ and control of $m_L$.} We define a new variable 
	\begin{align*}
		w=\nabla(\rho-\rho_\infty)+\gamma (m-m_{L}),
	\end{align*}	
	where $m_{L}(x,t)$ solves the linear damped heat equation
	\begin{equation*}
		\left\{
		\begin{aligned}
			&\partial_tm_{L}-\gamma^{-1}\Delta m_{L}+\gamma m_{L}=0,\\	
			&m_{L}(x,0)=m_0(x).
		\end{aligned}
		\right.
	\end{equation*}
By multiplying by $m_L$ in $H^3$ and integrating by parts, we obtain the energy identity
	\begin{align}\label{030701}
		\|m_{L}\|_{H^3}^2+2\int_0^t(\gamma^{-1}\|\nabla m_{L}(s)\|_{H^3}^2+\gamma\|m_{L}(s)\|_{H^3}^2)ds=\|m_{0}\|_{H^3}^2.
	\end{align}
	
\medskip

\noindent \emph{Step 2: Evolution of $w$ and energy inequality.} Rewrite the momentum equation $\eqref{MF}_2$ as  
	\[
	\partial_t m+\nabla(\rho-\rho_\infty)+\gamma m=-\text{div}\Big(\frac{m\otimes m}{\rho}\Big)-(\rho-\rho_\infty)\nabla V=-f.
	\]
	Denote $\bar{m} =m-m_{L}$, then we obtain 
\[
		w=\nabla(\rho-\rho_\infty)+\gamma\bar{m} =-(\partial_t\bar{m} +\gamma^{-1}\Delta m_{L}+f).
\]

	By standard energy estimates, we deduce
	\begin{align*}
		& \|w\|_{H^2}^2+\frac{1}{2}\gamma \frac{d}{dt}\left(
		\|\bar{m} \|_{H^2}^2+2\gamma^{-1}\sum_{k=0}^2 \int_{\R^3} \nabla^k\bar{m}  \cdot \nabla^k\nabla(\rho-\rho_\infty) dx\right)\\
		&\quad =-\sum_{k=0}^2  \int_{\R^3} \nabla^k\text{div}\bar{m}  \cdot \nabla^k\partial_t(\rho-\rho_\infty)dx - \sum_{k=0}^2  \int_{\R^3} (\nabla^kf+\gamma^{-1}\nabla^k\Delta m_{L}) \cdot \nabla^kw \, dx.
	\end{align*}
	Using the equation $\eqref{MF}_1$, we get 
	\begin{align*}
		-\sum_{k=0}^2  \int_{\R^3} \nabla^k\text{div}\bar{m}  \cdot \nabla^k\partial_t(\rho-\rho_\infty)dx &=\|\text{div}m\|_{H^2}^2-\sum_{k=0}^2  \int_{\R^3} \nabla^k\text{div}m_L \cdot \nabla^k\text{div}m\,dx\\
		&\leq C(\|m\|_{H^3}^2+\|m_{L}\|_{H^3}^2)
	\end{align*}
	and
	\begin{align*}
		& - \sum_{k=0}^2  \int_{\R^3} (\nabla^kf+\gamma^{-1}\nabla^k\Delta m_{L}) \cdot \nabla^kw \, dx\\
		&\quad \leq \frac{1}{2}\|w\|_{H^2}^2+C\|f\|_{H^2}^2
		+C\gamma^{-2}\|\nabla m_{L}\|_{H^3}^2\\
		&\quad \leq \frac{1}{2}\|w\|_{H^2}^2+C(\|m\|_{H^3}^2+\|u\|_{H^3}^2+  \epsilon_0\|\nabla(\rho-\rho_\infty)\|_{L^2})^2
		+C\gamma^{-2}\|\nabla m_{L}\|_{H^3}^2,
	\end{align*}
	where we used the fact that 
	\begin{align*}
		\|f\|_{H^2}&\leq C \|m\|_{H^3}\|u\|_{H^3}+C\|\nabla(\rho-\rho_\infty)\|_{H^2}\|\nabla V\|_{H^3}\\	
		&\leq C(\|m\|_{H^3}^2+\|u\|_{H^3}^2+\epsilon_0\|\nabla(\rho-\rho_\infty)\|_{H^2}).
	\end{align*}
This gives
	\begin{align*}
		& \|w\|_{H^2}^2+\gamma \frac{d}{dt}\left(
		\|\bar{m} \|_{H^2}^2+2\gamma^{-1}\sum_{k=0}^2 \int_{\R^3} \nabla^k\bar{m}  \cdot \nabla^k\nabla(\rho-\rho_\infty) dx\right)\\
		&\quad \lesssim \|m_{L}\|_{H^3}^2+\gamma^{-2}\|\nabla m_{L}\|_{H^3}^2+\|f\|_{H^2}^2+\|m\|_{H^3}^2\\
		&\quad \lesssim \|m_{L}\|_{H^3}^2+\gamma^{-2}\|\nabla m_{L}\|_{H^3}^2+\epsilon_0^2\|\nabla(\rho-\rho_\infty)\|_{H^2}^2+(\|m\|_{H^3}^2+\|u\|_{H^3}^2)^2
		+\|m\|_{H^3}^2.
	\end{align*}
	Integrating the above equation with respect to time and using \eqref{030701} yield that 
	\begin{align*}
		& \int_0^t\|w(s)\|_{H^2}^2ds+\gamma \left(
		\|\bar{m} \|_{H^2}^2+2\gamma^{-1}\sum_{k=0}^2 \int_{\R^3} \nabla^k\bar{m}  \cdot \nabla^k\nabla(\rho-\rho_\infty) dx\right)\\
		&\quad \leq C\int_0^t\Big((\|m(s)\|_{H^3}^2+\|u(s)\|_{H^3}^2)^2+\|m(s)\|_{H^3}^2\Big)ds+C\gamma^{-1}\|m_0\|_{H^3}^2\\
		&\qquad +C\gamma^{-1}\int_0^t(\gamma \|m_{L}(s)\|_{H^3}^2+\gamma^{-1}\|\nabla m_{L}(s)\|_{H^3}^2)ds+C\epsilon_0^2\int_0^t \|\nabla(\rho-\rho_\infty)(s)\|_{H^2}^2ds\\
		&\quad \leq C\int_0^t\Big((\|m(s)\|_{H^3}^2+\|u(s)\|_{H^3}^2)^2+\|m(s)\|_{H^3}^2\Big)ds
		+C\gamma^{-1}\|m_0\|_{H^3}^2\nonumber\\
		&\qquad+C\epsilon_0^2\int_0^t\|\nabla(\rho-\rho_\infty)(s)\|_{H^2}^2ds.
	\end{align*}
	Noting that 
	\begin{align*}
		\sum_{k=0}^2\int_{\R^3} \nabla^k\bar{m}  \cdot \nabla^k\nabla(\rho-\rho_\infty) dx \leq \frac{1}{4}\gamma \|\bar{m} \|_{H^2}^2+\gamma^{-1}\|\nabla(\rho-\rho_\infty)\|_{H^2}^2
	\end{align*}
	and 
	\[
	\gamma \|\bar{m} \|_{H^2}^2+2\sum_{k=0}^2\int_{\R^3} \nabla^k\bar{m}  \cdot \nabla^k\nabla(\rho-\rho_\infty) dx\geq \frac{1}{2}\gamma \|\bar{m} \|_{H^2}^2-2\gamma^{-1}\|\nabla(\rho-\rho_\infty)\|_{H^2}^2.
	\]
Thus, we obtain 
	\begin{align}\label{031301}
	\begin{aligned}
		& \int_0^t\|w(s)\|_{H^2}^2ds+\gamma \|\bar{m} (t)\|_{H^2}^2\\
		&\quad \leq C\int_0^t\Big((\|m(s)\|_{H^3}^2+\|u(s)\|_{H^3}^2)^2+\|m(s)\|_{H^3}^2 \Big)ds		+C\gamma^{-1}\|m_0\|_{H^3}^2\\
		&\qquad+C\epsilon_0^2\int_0^t\|\nabla(\rho-\rho_\infty)(s)\|_{H^2}^2ds+C\gamma^{-1}\|\nabla(\rho-\rho_\infty)\|_{H^2}^2\\
		&\quad \leq C\gamma\delta^4+C\delta^2+ C\gamma^{-1}(B_0^2+B_0)^2+ C\gamma \epsilon_0^2 \mathcal{R}_0^2,
\end{aligned}
	\end{align}
	due to \eqref{030602}.

\medskip
\noindent \emph{Step 3: Relating $\rho$ to $\rho_*$ and $m_L$.} Define the parabolic comparison $\rho_*$ solving \eqref{H1} with the same initial data. By Lemma \ref{lemM},
	\begin{align}\label{Max}
		\rho_1\leq \rho_*\leq \rho_2.
	\end{align}
	By the definition of $w$, we get
	\[
	\mathrm{div}m = \mathrm{div}\bar{m} +\mathrm{div}m_{L}=\gamma^{-1}\mathrm{div}w - \gamma^{-1}\Delta(\rho-\rho_\infty)+\mathrm{div}m_{L}.
	\]
	This relation together with the continuity equation $\eqref{MF}_1$ yields 
	\[
	\partial_{t}(\rho-\rho_\infty)-\gamma^{-1}\Delta(\rho-\rho_\infty)=-\gamma^{-1}\mathrm{div}w-\mathrm{div}m_{L}.
	\]
	Then we find
	\[
	-\mathrm{div}m_{L}=\partial_{t}(\gamma^{-1}\mathrm{div}m_{L})-\gamma^{-1}\Delta(\gamma^{-1}\mathrm{div}m_{L}),
	\]
and thus
	\begin{align}\label{rhogamma}
	\partial_{t}(\rho-\rho_\infty-\gamma^{-1}\mathrm{div}m_{L})-\gamma^{-1}\Delta(\rho -\rho_\infty-\gamma^{-1}\mathrm{div}m_{L})=-\gamma^{-1}\mathrm{div}w.
	\end{align}
	Let $  \bar{\rho} =\rho-\rho_{*}-\gamma^{-1}\mathrm{div}m_{L}$. Then, combining \eqref{rhogamma} and $\eqref{H1}_1$ gives
	\[
	\begin{cases}
		\partial_{t}  \bar{\rho} -\gamma^{-1}\Delta  \bar{\rho} =-\gamma^{-1}\mathrm{div}w-\gamma^{-1}\text{div}((\rho_*-\rho_\infty)\nabla V),\\
		  \bar{\rho} (x,0)=-\gamma^{-1}\mathrm{div}m_{0}.
	\end{cases}
	\]
	Taking $H^2$ energy estimates, we obtain
	\begin{align*}
		&\frac{d}{dt}\|  \bar{\rho} \|_{H^{2}}^{2}+2\gamma^{-1}\|\nabla  \bar{\rho} \|_{H^{2}}^{2}\cr
		&\quad =-2\gamma^{-1}\sum_{k=0}^2\int_{\R^3}\nabla^k  \bar{\rho}  \cdot \nabla^k\mathrm{div}w \,dx -2\gamma^{-1}\sum_{k=0}^2 \int_{\R^3} \nabla^k  \bar{\rho}  \cdot \nabla^k\text{div}((\rho_*-\rho_\infty)\nabla V)dx\\
			&  \quad =2\gamma^{-1}\sum_{k=0}^2\int_{\R^3}\nabla^{k+1}\bar{\rho}  \cdot \nabla^kw \,dx +2\gamma^{-1}\sum_{k=0}^2 \int_{\R^3} \nabla^{k+1}\bar{\rho}  \cdot \nabla^k((\rho_*-\rho_\infty)\nabla V)dx\\	
		&\quad \leq \gamma^{-1}\|\nabla  \bar{\rho} \|_{H^2}^2+C\gamma^{-1}\|w\|_{H^2}^2
		+C\gamma^{-1}\|(\rho_*-\rho_\infty)\nabla V\|_{H^2}^2\\
		&\quad \leq \gamma^{-1}\|\nabla   \bar{\rho} \|_{H^2}^2+C\gamma^{-1}\|w\|_{H^2}^2+C\gamma^{-1}\epsilon_0^2\|\nabla(\rho_*-\rho_\infty)\|_{H^2}^2.
	\end{align*}
	Integrating with respect to time  and using \eqref{031301} yields  
	\begin{align*}
		&\|  \bar{\rho} (t)\|_{H^2}^2+\gamma^{-1}\int_0^t\|\nabla  \bar{\rho} (s)\|_{H^2}^2ds\\
		&\quad \leq C\gamma^{-1}\int_0^t\|w(s)\|_{H^2}^2ds+C\gamma^{-1}\epsilon_0^2B_0^2+C\gamma^{-2}\|m_0\|_{H^3}^2	+\|  \bar{\rho} _0\|_{H^2}^2\\
		&\quad \leq C\delta^4+C\gamma^{-1}\delta^2+C\gamma^{-2} (B_0^2+B_0)^2 +  C\epsilon_0^2 \mathcal{R}_0^2+C\gamma^{-1}\epsilon_0^2B_0^2.
	\end{align*}

\medskip
\noindent \emph{Step 4: Conclusion by Sobolev embedding.} By Sobolev embedding $H^2(\R^3) \hookrightarrow L^\infty(\R^3)$, we get
	\begin{align*}
		\|\rho-\rho_*\|_{L^\infty}\lesssim \|\rho-\rho_*\|_{H^2}\leq  C(\delta^4+\gamma^{-1}\delta^2+\gamma^{-2} (B_0^2+B_0)^2 +  \epsilon_0^2 \mathcal{R}_0^2 +\gamma^{-1}\epsilon_0^2B_0^2)^{\frac{1}{2}}.
	\end{align*}
This together with \eqref{Max} yields 
	\begin{align*}
		&\rho_{1}-C(\delta^4+\gamma^{-1}\delta^2+\gamma^{-2} (B_0^2+B_0)^2 +  \epsilon_0^2 \mathcal{R}_0^2+\gamma^{-1}\epsilon_0^2B_0^2)^{\frac{1}{2}}\\
		&\quad \leq\rho(x,t)\leq\rho_{2}+C(\delta^4+\gamma^{-1}\delta^2+\gamma^{-2} (B_0^2+B_0)^2 +  \epsilon_0^2 \mathcal{R}_0^2 +\gamma^{-1}\epsilon_0^2B_0^2)^{\frac{1}{2}}.
	\end{align*}
 	Due to the smallness of $\gamma^{-1}$, $\delta$, and $\epsilon_0$,  and the boundedness of $B_0$ , we conclude
\[
		\frac{1}{2}\rho_1\leq \rho(x,t)\leq \frac{3}{2}\rho_2.
\]
\end{proof}

%
%
%
%
%
%
%
%
%
%
%
%
%

\subsection{ High-order energy estimates}

With the uniform $L^\infty$ bound on the density now established, the next step is to derive global-in-time  \emph{a priori} energy estimates for the solution. These estimates are crucial for closing the bootstrap assumption \eqref{a-priori} and ensuring that the local solution can be extended globally. To this end, we establish the following proposition, which provides uniform bounds for the Sobolev norms of $(\rho, m, u)$ and their derivatives.

\begin{proposition}\label{p31}
	Assume that all the conditions of Theorem \ref{thm1} hold, and let $(\rho, u)(x,t)$ be the regular solution to system \eqref{MF1} satisfying the a priori assumption \eqref{a-priori}.   Then there exist constants $C_0=C(B_0)$ and $\mathcal{R}_0 = \mathcal{R}_0(B_0)$, independent of $t$, such that for all $t>0$ 
\[
	\|(\rho-\rho_\infty,m,u)(t)\|_{H^3}^2+\gamma\int_0^t \|(m, u)(s)\|_{H^3}^2 ds\leq C_0\quad \mbox{and} \quad \gamma^{-1}
	\int_0^t\|\nabla \phi(s)\|_{H^2}^2ds\leq \frac{1}{2}\mathcal{R}_0^2.
\]
\end{proposition}
\begin{proof}
We divide the proof into two steps.

\medskip
\noindent \emph{Step 1: Low-order energy estimate.} We first derive the $H^2$-energy estimate for $(\rho-\rho_\infty,m)$. Multiplying $\nabla^k(\rho-\rho_\infty)$ and $\nabla^k m$ to the first and second equations of \eqref{MF}, respectively, integrating over $\R^3$, and summing over $k=0,1,2$, we obtain
	\begin{align*}
		&\frac{1}{2}\frac{d}{dt}\|(\rho-\rho_\infty, m)\|_{H^2}^2+\gamma \|m\|_{H^2}^2\\
		&\quad \leq C\|m\|_{H^2}\|\nabla(\rho-\rho_\infty)\|_{H^1}\|\nabla V\|_{H^2}+C\|u\|_{H^3}\|m\|_{H^3}\|m\|_{H^2}\\
		&\quad \leq  \frac12  \gamma\|m\|_{H^2}^2+C\gamma^{-1}\|\nabla(\rho-\rho_\infty)\|_{H^2}^2\|\nabla V\|_{H^2}^2
		+C\gamma^{-1}(\|m\|_{H^3}^2+\|u\|_{H^3}^2)^2\\
		&\quad \leq  \frac12 \gamma\|m\|_{H^2}^2+C\gamma^{-1}\epsilon_0^2\|\nabla(\rho-\rho_\infty)\|_{H^2}^2
		+C\gamma^{-1}(\|m\|_{H^3}^2+\|u\|_{H^3}^2)^2.
	\end{align*}
Here, we used the following estimate for the nonlinear flux term: 
	\begin{align*}
	\left\|\text{div}(\frac{m\otimes m}{\rho})\right\|_{H^2}=\|\text{div}(u\otimes m)\|_{H^2}\leq C\|u\|_{H^3}\|m\|_{H^3}.
	\end{align*}
 This gives 
\[
		\frac{d}{dt}\|(\rho-\rho_\infty, m)\|_{H^2}^2+\gamma \|m\|_{H^2}^2  \leq C\gamma^{-1}\epsilon_0^2\|\nabla(\rho-\rho_\infty)\|_{H^2}^2
		+C\gamma^{-1}(\|m\|_{H^3}^2+\|u\|_{H^3}^2)^2.
\]
	Integrating the above inequality with respect to time and using the \emph{a priori} bound \eqref{a-priori} give
	\begin{align}\label{030604}
	\begin{aligned}
		&\|(\rho-\rho_\infty, m)\|_{H^2}^2+\gamma \int_0^t\|m(s)\|_{H^2}^2ds \\
		&\quad \leq C\|(\rho_0-\rho_\infty,m_0)\|_{H^2}^2+C\gamma^{-1}\epsilon_0^2\int_0^t\|\nabla(\rho-\rho_\infty)(s)\|_{H^2}^2ds \\
		&\qquad +C\gamma^{-1}\int_0^t(\|m(s)\|_{H^3}^2+\|u(s)\|_{H^3}^2)^2ds \\
		&\quad \leq C( (B_0^2+B_0)^2 +B_0^2)+ C\epsilon_0^2\mathcal{R}_0^2 + C\delta^4 \\
		&\quad \leq   C(B_0^4+B_0^2) ,
		\end{aligned}
	\end{align}
owing to the smallness of $\delta, \epsilon_0$ and \eqref{030602}.

\medskip
\noindent \emph{Step 2: High-order energy estimate.} To control the high-order derivatives, we employ the velocity formulation of the system \eqref{Main1}. Set $\phi=\ln \rho-\ln \rho_\infty$ and $\phi_0=\ln \rho_0-\ln \rho_\infty$ and rewrite \eqref{Main1} in terms of $\phi$ and $u$:
\begin{equation*}
 	\left\{
 	\begin{aligned}	
 		&\partial_t\phi+\text{div}u=-u\cdot\nabla \phi+u\cdot\nabla V,\\
 		&\partial_tu+\nabla \phi+\gamma u=-u\cdot\nabla u.
 	\end{aligned}
 	\right.
 \end{equation*}
Differentiating up to third order and applying energy estimates, commutator bounds (Lemma \ref{lema2}), and a hypocoercivity-type argument, we obtain
	\begin{equation}\label{ty1}
		\|\nabla(\phi,u)(t)\|_{H^2}^2+\gamma\int_0^t\|\nabla u(s)\|_{H^2}^2ds  \leq C_0 \quad \mbox{and} \quad \gamma^{-1}
	\int_0^t\|\nabla \phi(s)\|_{H^2}^2ds\leq \frac{1}{2}\mathcal{R}_0^2
	\end{equation}
for some constants  $C_0=C(B_0) > 0$ and $\mathcal{R}_0 = \mathcal{R}_0(B_0)>0$. Since the derivation of \eqref{ty1} involves lengthy commutator estimates and higher-order nonlinear bounds, we postpone its full proof to Appendix \ref{app:high-order} for clarity of presentation.
	
Moreover, since $\phi=\ln\rho-\ln \rho_\infty$, it follows that 
	\begin{align}\label{031102}
		\|\nabla(\rho-\rho_\infty)\|_{H^2}=\|\nabla(\rho_\infty e^{\phi}-\rho_\infty)\|_{H^2}\leq C\|\nabla \phi\|_{H^2}.	
	\end{align}
Finally,  combining \eqref{ty1} and \eqref{031102},  we conclude
	\begin{align*}
	\|(\rho-\rho_\infty,m,u)(t)\|_{H^3}^2+\gamma\int_0^t \|(m, u)(s)\|_{H^3}^2 ds\leq C_0,
	\end{align*}
where $C_0=C(B_0)$ denotes a positive constant independent of time. 
This completes the proof.
\end{proof}


\subsection{Proof of Theorem \ref{thm1}.} 
We are now in a position to complete the proof of Theorem \ref{thm1}. 

Proposition \ref{p31} provides a uniform-in-time \emph{a priori} bound for the solution to \eqref{MF1} in the high-order Sobolev space $H^3$,
\begin{align}\label{TR1}
	\|(\rho-\rho_\infty,m,u)(t)\|_{H^3}^2+\gamma\int_0^t \|(m, u)(s)\|_{H^3}^2 ds\leq C_0, \quad \gamma^{-1}
	\int_0^t\|\nabla \phi(s)\|_{H^2}^2ds\leq \frac{1}{2}\mathcal{R}_0^2,
\end{align}
where $C_0$ and $\mathcal{R}_0$ depend only on the initial data.

To close the \emph{a priori} assumption \eqref{a-priori}, we recall that $Z_1(t)$ and $Z_2(t)$ were defined as scaled energy functionals combining both instantaneous $H^3$ norms and time-integrated dissipation. By choosing $\gamma^{-1}$ and $\epsilon_0$ sufficiently small, the above estimate immediately yields 
\begin{align*}
Z_1(t)^2 
	&\leq 2\gamma^{-1}\|(\rho-\rho_\infty,m,u)(t)\|_{H^3}^2 + 2\int_0^t   \|(m,u)(s)\|_{H^3}^2 \,ds  \leq C\gamma^{-1} C_0	\leq \frac{1}{4}\delta^2
\end{align*}
and 
	\begin{align*}
Z_2(t)^2=\gamma^{-1}\int_0^t\|\nabla(\rho-\rho_\infty)(s)\|_{H^2}^2ds\leq \frac{1}{2}\mathcal{R}_0^2, 	
\end{align*}
which confirms that the bootstrap assumption \eqref{a-priori} is indeed valid for all $t>0$. 

Finally, combining this global a priori bound with the local-in-time existence and uniqueness result (Theorem \ref{lem-loc}), we apply the standard continuity argument to extend the local solution to a unique global classical solution. Thus, the global well-posedness of \eqref{Main1} is established for sufficiently large damping coefficients.

According to the Mean Value Theorem and Proposition \ref{A2}, one has
\[
	\phi=\ln \rho-\ln\rho_\infty=\frac{1}{\rho_y}(\rho-\rho_\infty),\quad \rho_y\in (\rho_1,\rho_2).	
\]
Then we deduce from \eqref{TR1} that 
\begin{align*}
	\|(\phi,u)(t)\|_{H^3}^2
	+\int_0^t(\gamma\|u(s)\|_{H^3}^2+\gamma^{-1}\|\nabla\phi(s)\|_{H^2}^2)\,ds\leq C_0,
\end{align*}
where the positive constant $C_0=C(B_0)$ only depends on the initial data $B_0=\|(\phi_0,u_0)\|_{H^3}$. Therefore, we complete the proof of Theorem \ref{thm1}.

	\section{Global well-posedness and large-time behavior  for small initial data }\label{Sec:thm3-4}	
In this section, we establish the global-in-time well-posedness and derive the decay rates for solutions to the reformulated damped Euler system \eqref{Main31}, thereby providing the detailed proof of the existence part in Theorem \ref{thm3}.

\subsection{Global well-posedness for small data}	
To extend the local classical solution constructed in Theorem \ref{lem-loc} to a global one, 
we first derive uniform-in-time energy bounds. This requires a careful analysis of the nonlinear structure of the system and the introduction of suitable energy functionals. 
We begin by defining the basic energy and dissipation quantities, followed by deriving high-order energy inequalities and mixed estimates that are key to closing the bootstrap argument.

 \subsubsection{Energy framework and high-order energy inequality} 
 We introduce the energy functional $E(t)$ and its associated dissipation rate $D(t)$ by
\[
		E(t) :=\|\phi(t)\|_{H^3}+\|u(t)\|_{H^3}\quad \mbox{and} \quad D(t):=\|\nabla\phi(t)\|_{H^2}+\|u (t)\|_{H^3},
\]
	respectively.  To control the nonlinear terms, we make the following \emph{a priori} assumption:
	\begin{equation}\label{030802}
		E(t)=\|\phi(t)\|_{H^3}+\|u(t)\|_{H^3}\leq \delta,
	\end{equation}
where $\delta>0$ is a sufficiently small constant. This smallness will be used to close the energy estimates.

We first derive a differential inequality for the high-order energy.		
	\begin{lemma}\label{l1}
		Let $T>0$ and $(\phi,u)$ be a classical solution of system \eqref{Main31} satisfying the a priori bound \eqref{030802}. Then there eixsts $C>0$ independent of $t$ such that for $0<t\leq T$ 
		\begin{align}\label{010703}
			\frac{1}{2}\frac{d}{d t}\|(\phi, u)\|_{H^3}^2+\gamma \|u\|_{H^3}^2\leq CE(t)D(t)^2+C\delta_0 D(t)^2.
		\end{align}
	\end{lemma}
	\begin{proof}
We begin with the basic $L^2$ energy estimate for \eqref{Main31}:
		\begin{align*}
			\frac{1}{2}\frac{d}{dt}\|(\phi,u)\|_{L^2}^2+\gamma\|u\|_{L^2}^2
			=\int_{\mathbb{R}^3}(-u\cdot\nabla\phi+ u\cdot\nabla V)\phi dx
			+\int_{\mathbb{R}^3}(-u\cdot\nabla u)\cdot udx.  
		\end{align*}	
Using H\"{o}lder's inequality and Sobolev embeddings, we estimate 
		\begin{align*}
			\left| \int_{\mathbb{R}^3}(-u\cdot\nabla\phi+ u\cdot\nabla V)\phi dx  \right|
			&\leq	\|u\|_{L^3}\|\nabla\phi\|_{L^2}\|\phi\|_{L^6}+ \|u\|_{L^2}\|\nabla V\|_{L^3}\|\phi\|_{L^6}\nonumber\\
			&\leq C\|u\|_{H^1}\|\nabla \phi\|_{L^2}^2
			+C\|u\|_{L^2}\|\nabla V\|_{H^1}\|\nabla\phi\|_{L^2}\nonumber\\
			&\leq CE(t)D(t)^2+C\delta_0 D(t)^2
		\end{align*}
		and 
		\begin{align*}
			\left|\int_{\mathbb{R}^3}(-u\cdot\nabla u)\cdot udx\right|\leq \|u\|_{L^3}\|\nabla u\|_{L^2}\|u\|_{L^6}\leq C\|u\|_{H^1}\|\nabla u\|_{L^2}^2\leq CE(t)D(t)^2.
		\end{align*}
Thus,
		\begin{align}\label{010705}
			\frac{1}{2}\frac{d}{dt}\|(\phi,u)\|_{L^2}^2+\gamma \|u\|_{L^2}^2\leq CE(t)D(t)^2+C\delta_0 D(t)^2.	
		\end{align}

Next, for the high-order derivatives $k\geq 1$, we obtain
		\begin{align*}
			&\frac{1}{2}\frac{d}{dt}\|\nabla^k(\phi, u)\|_{L^2}^2+\gamma\|\nabla^k u\|_{L^2}^2\nonumber\\
			&\quad =\int_{\mathbb{R}^3}\nabla^k(-u\cdot\nabla\phi+ u\cdot\nabla V)\cdot\nabla^k\phi dx
			+\int_{\mathbb{R}^3}\nabla^k(-u\cdot\nabla u)\cdot \nabla^k udx.
		\end{align*}	
We decompose
		\begin{align*}
			&\int_{\mathbb{R}^3}\nabla^k(-u\cdot\nabla\phi+ u\cdot\nabla V)\cdot\nabla^k\phi dx\nonumber\\
			&\quad =-\int_{\mathbb{R}^3}(\nabla^k(u\cdot\nabla\phi)-u\cdot\nabla^k\nabla\phi)\cdot\nabla^k\phi dx-\int_{\mathbb{R}^3}u\cdot\nabla^k\nabla\phi\cdot\nabla^k\phi dx\nonumber\\
			&\qquad +\int_{\mathbb{R}^3}(\nabla^k(u\cdot\nabla V)-u\cdot\nabla^k\nabla V)\cdot\nabla^k\phi dx+\int_{\mathbb{R}^3}u\cdot\nabla^k\nabla V\cdot\nabla^k\phi dx.
		\end{align*}
		By the communicator estimate in Lemma \ref{lema2}, we find
		\begin{align}\label{010701}
		\begin{aligned}
			&-\int_{\mathbb{R}^3}(\nabla^k(u\cdot\nabla\phi)-u\cdot\nabla^k\nabla\phi)\cdot\nabla^k\phi dx-\int_{\mathbb{R}^3}u\cdot\nabla^k\nabla\phi\cdot\nabla^k\phi dx \\
			&\quad  \leq \|\nabla^k(u\cdot\nabla\phi)-u\cdot\nabla^k\nabla\phi\|_{L^2}\|\nabla^k\phi\|_{L^2}
			+C\|\text{div}u\|_{L^\infty}\|\nabla^k\phi\|_{L^2}^2 \\
			&\quad  \leq C(\|\nabla u\|_{L^\infty}\|\nabla^k\phi\|_{L^2}+\|\nabla^ku\|_{L^2}\|\nabla\phi\|_{L^\infty})\|\nabla^k\phi\|_{L^2}+C\|\nabla^2u\|_{H^1}\|\nabla^k\phi\|_{L^2}^2 \\
			&\quad  \leq C(\|\nabla^2 u\|_{H^1}\|\nabla^k\phi\|_{L^2}+\|\nabla^ku\|_{L^2}\|\nabla^2\phi\|_{H^1})\|\nabla^k\phi\|_{L^2}+C\|\nabla^2u\|_{H^1}\|\nabla^k\phi\|_{L^2}^2 \\
			&\quad  \leq CE(t)D(t)^2.
			\end{aligned}
		\end{align}	
		Similarly, we also obtain for $k=1,2,3$ 
		\begin{align*}
			&\int_{\mathbb{R}^3}(\nabla^k(u\cdot\nabla V)-u\cdot\nabla^k\nabla V)\cdot\nabla^k\phi dx+\int_{\mathbb{R}^3}u\cdot\nabla^k\nabla V\cdot\nabla^k\phi dx\nonumber\\
			&\quad \leq C(\|\nabla u\|_{L^\infty}\|\nabla^kV\|_{L^2}+\|\nabla^ku\|_{L^2}\|\nabla V\|_{L^\infty})\|\nabla^k\phi\|_{L^2}+C\|u\|_{L^\infty}\|\nabla^{k+1}V\|_{L^2}\|\nabla^k\phi\|_{L^2}\nonumber\\
			&\quad \leq C(\|\nabla^2 u\|_{H^1}\|\nabla^kV\|_{L^2}+\|\nabla^ku\|_{L^2}\|\nabla^2 V\|_{H^1})\|\nabla^k\phi\|_{L^2}+C\|\nabla u\|_{H^1}\|\nabla^{k+1}V\|_{L^2}\|\nabla^k\phi\|_{L^2}\nonumber\\
			&\quad \leq C\delta_0 D(t)^2.
		\end{align*}
This together with \eqref{010701} deduces
		\begin{align*}
			\int_{\mathbb{R}^3}\nabla^k(-u\cdot\nabla\phi+ u\cdot\nabla V)\cdot\nabla^k\phi dx
			\leq CE(t)D(t)^2+C\delta_0 D(t)^2.
		\end{align*}
The nonlinear convection term can be estimated as
		\begin{align*}
			\int_{\mathbb{R}^3}\nabla^k(-u\cdot\nabla u)\cdot \nabla^k udx\leq 
			CE(t)D(t)^2.
		\end{align*}
		Therefore we conclude from above that 
		\begin{align*}
			\frac{1}{2}\frac{d}{dt}\|\nabla^k(\phi, u)\|_{L^2}^2+\gamma\|\nabla^k u\|_{L^2}^2
			\leq CE(t)D(t)^2+C\delta_0 D(t)^2.
		\end{align*}	
Summing over $k=1,2,3$ and combining with \eqref{010705} yields \eqref{010703}.
	\end{proof}
	
%
%
%
%
%
%
%
%
\subsubsection{Hypocoercivity-type estimate}
We now establish a dissipation mechanism for the density fluctuation, which is essential for controlling the coupled dynamics of $(\phi,u)$ and closing the global energy estimates.
To this end, we introduce a mixed energy functional incorporating carefully chosen cross terms between the velocity and density variables.
This hypocoercivity-type argument recovers dissipation for $\phi$ from the pressure-velocity coupling, which is not directly provided by the damping term.
	\begin{lemma}\label{p2}
		Let $T>0$ and $(\phi, u)$ be a classical solution of system \eqref{Main31} satisfying the a priori bound \eqref{030802}. Then there exist positive constants $C$ and $\tilde{C}_*$ independent of time such that for $0<t\leq T$ and  $k=0,1,2$
		\begin{align}\label{010704}
			\frac{d}{dt} \sum_{k=0}^2 \int_{\R^3} \nabla^k u \cdot \nabla^{k+1}\phi\,dx +\frac{1}{2}\|\nabla \phi\|_{H^2}^2
			\leq CE(t)D(t)^2+C\delta_0 D(t)^2+\tilde{C}_*\| u\|_{H^3}^2.
		\end{align}
	\end{lemma}
	\begin{proof}
		Taking $L^2$ energy estimate on \eqref{Main31} with $\nabla \phi$, we get
		\begin{align*}
			&\frac{d}{dt} \int_{\R^3} u \cdot \nabla\phi\,dx+\|\nabla\phi\|_{L^2}^2
			\cr
			&\quad = \int_{\R^3} u \cdot \nabla(-\text{div}u+f_1)\,dx - \gamma \int_{\R^3} u\cdot \nabla\phi\,dx -  \int_{\R^3} (u\cdot\nabla u) \cdot \nabla\phi\,dx  \\
			&\quad  \leq  \|\text{div}u\|_{L^2}^2+ \int_{\R^3} u \cdot \nabla f_1\,dx + \gamma \|u\|_{L^2}\|\nabla\phi\|_{L^2}
			+\|u\|_{L^6}\|\nabla u\|_{L^2}\|\nabla\phi\|_{L^3}\nonumber\\
			&\quad \leq \|\nabla u\|_{L^2}^2+\frac{1}{2}\|\nabla \phi\|_{L^2}^2+\frac{1}{2}\gamma^2\|u\|_{L^2}^2			+C\|\nabla u\|_{L^2}^2\|\nabla\phi\|_{H^1}\nonumber\\
			&\quad \leq \tilde{C}_1\|u\|_{H^1}^2+\frac{1}{2}\|\nabla\phi\|_{L^2}^2+CE(t)D(t)^2,
		\end{align*}
		where the positive constant $\tilde{C}_1=1+\frac{1}{2}\gamma^2$ and we used  
		\begin{align*}
			\int_{\R^3} u \cdot \nabla f_1\,dx &= - \int_{\R^3} \text{div}u \cdot f_1\,dx \cr
			&=-\int_{\R^3} \text{div}u \cdot (-u\cdot\nabla\phi+u\cdot\nabla V)\,dx \\	
			&\leq \|\text{div}u\|_{L^2}(\|u\|_{L^6}\|\nabla\phi\|_{L^3}+\|u\|_{L^3}\|\nabla  V \|_{L^6})\nonumber\\
			&\leq C\|\nabla u\|_{L^2}(\|\nabla u\|_{L^2}\|\nabla \phi\|_{H^1}+\|u\|_{H^1}\|\nabla^2V\|_{L^2})\nonumber\\
			&\leq C\|\nabla u\|_{L^2}(\|\nabla u\|_{L^2}\|\nabla \phi\|_{H^1}+\delta_0 \|u\|_{H^1})\nonumber\\
			&\leq CE(t)D(t)^2+C\delta_0 D(t)^2.
		\end{align*}
Combining these estimates gives
		\begin{align*}
			\frac{d}{dt} \int_{\R^3} u \cdot \nabla\phi\,dx + \frac{1}{2}\|\nabla\phi\|_{L^2}^2
			\leq CE(t)D(t)^2+C\delta_0 D(t)^2+\tilde{C}_1\|u\|_{H^1}^2.
		\end{align*}

Repeating the argument for higher derivatives introduces similar commutator terms, which can be controlled by the same techniques. Thus, for each $k=1,2$, there exist positive constants $\tilde{C}_2$ and $\tilde{C}_3$ such that
		\begin{align*}
			\frac{d}{dt}\int_{\mathbb{R}^3}\nabla^k u\cdot\nabla^{k+1}\phi dx+\frac{1}{2}\|\nabla^{k+1}\phi\|_{L^2}^2
			\leq CE(t)D(t)^2+C\delta_0 D(t)^2+ (\tilde{C}_2+\tilde{C}_3)\|\nabla^{k+1} u\|_{H^1}^2.
		\end{align*}
	
 		Summing over $k=0,1,2$ and setting $\tilde{C}_* = \tilde{C}_1 + \tilde{C}_2 + \tilde{C}_3$, we obtain \eqref{010704}. This completes the proof.
	\end{proof}
		\vskip 3mm	
	%
	%
	%
	%
	%
	%
	%
	%
\subsubsection{Proof of Theorem \ref{thm3}: global well-posedness.} 
We now close the bootstrap argument and establish the global-in-time well-posedness of classical solutions to \eqref{Main2}. 
The key step is to combine the basic high-order energy estimate in Lemma \ref{l1} with the cross-term estimate in Lemma \ref{p2}, which provides the missing dissipation for the density variable and allows us to recover full control over both velocity and density fluctuations.
 	
	Adding \eqref{010703} and $\kappa_1 \times$\eqref{010704} for a suitably small $\kappa_1>0$ yields
	\begin{equation}\label{030814}
		\begin{aligned}
			& \frac{d}{d t}\left(\|(\phi,u)\|_{H^3}^2+\kappa_1\sum_{k=0}^2\int_{\R^3} \nabla^k u \cdot \nabla^{k+1}\phi\,dx \right)+\gamma \|u\|_{H^3}^2+\frac{1}{2} \kappa_1 \|\nabla \phi\|_{H^2}^2 \\
			&\quad \leq C\kappa_1 E(t)D(t)^2+C\delta_0 \kappa_1D(t)^2+\kappa_1 \tilde{C}_*\|u\|_{H^3}^2
			+CE(t)D(t)^2+C\delta_0 D(t)^2.
		\end{aligned}
	\end{equation}
By taking $\kappa_1$ sufficiently small, we find 
	\begin{align*}
		\|(\phi,u)(t)\|_{H^3}^2+\kappa_1\sum_{k=0}^2\int_{\R^3} \nabla^k u \cdot \nabla^{k+1}\phi\,dx \sim \|(\phi, u)(t)\|_{H^3}^2
	\end{align*}
	Combining this with \eqref{030802} and \eqref{030814} shows 
	\begin{align*}
		\frac{d}{d t}\|(\phi, u)\|_{H^3}^2
		+\tilde{C}_4(\|\nabla \phi\|_{H^2}^2+\|u\|_{H^3}^2)\leq CE(t)D(t)^2+C\delta_0 D(t)^2\leq C(\delta+\delta_0)D(t)^2,
	\end{align*}
	where $\tilde{C}_4>0$ is a constant independent of time. Moreover, due to the smallness of $\delta_0$ and $\delta$, the right-hand side becomes negligible compared to the dissipation, giving
	\begin{align}\label{A3}
		\frac{d}{d t}\|(\phi, u)\|_{H^3}^2
		+\frac{1}{2}\tilde{C}_4(\|\nabla \phi\|_{H^2}^2+\|u\|_{H^3}^2)\leq 0.
	\end{align}
Integrating \eqref{A3} over $[0,t]$ and using the smallness of the initial data,
	\begin{align*}
		E(t)^2+\frac{1}{2}\tilde{C}_4 \int_0^tD(\tau)^2d \tau  \leq C\left\|\left(\phi_0, u_0\right)\right\|_{H^3}^2\leq C\delta_0^2\leq \frac{1}{4}\delta^2,
	\end{align*}
	which closes the \emph{a priori} bound \eqref{030802}. 
	
	Thus, by combining the local existence result (Theorem \ref{lem-loc}) with the continuation argument, we obtain the global existence and uniqueness of classical solutions to \eqref{Main2}, together with the energy bound \eqref{main-est}.

	%
	%
	%
	%
	%
	%
	%
	%
	%
	%

\subsection{Large-time decay of solutions}	
We next analyze the large-time behavior of solutions to the Euler system with damping, aiming to rigorously establish the decay rates stated in Theorem \ref{thm3}.
The key observation is that the damping term $\gamma u$ induces strong dissipation, which dominates the dynamics at large times. However, due to the hyperbolic structure of the underlying Euler system, the decay mechanism is frequency-dependent: low-frequency modes decay diffusively, while high-frequency modes experience rapid exponential damping. To capture this dichotomy, we employ a low-high frequency decomposition combined with spectral analysis of the linearized operator.

Specifically, we decompose the solution $W=(\phi,u)^{\top}$ into low- and high-frequency parts:	\begin{align}\label{011404}
		W(x,t)=W^{\ell}(x,t)+W^{h}(x,t)=:(\phi^{\ell},u^{\ell})+(\phi^{h},u^{h}),
	\end{align}
	where $W^{\ell}(x,t):=\mathcal{K}_1W(x,t)$ denotes the low-frequency component and $W^{h}(x,t):=\mathcal{K}_{\infty}W(x,t)$ represents the high-frequency component, as introduced in  \eqref{THZ5}. This decomposition allows us to separately treat the diffusive low-frequency modes and the strongly damped high-frequency modes, leading to sharp decay estimates. The detailed spectral analysis and sharp decay rates for the linearized system,
which play a key role in proving the optimality of the nonlinear decay rates,
will be carried out in Section \ref{sec:opt-decay}.
	
	%
	%
	%
	%
	%
\subsubsection{Spectral analysis of the linearized system.} To rigorously justify the decay mechanism described above, we first analyze the spectral structure of the linearized operator. 
This step clarifies how the damping and hyperbolic transport interact across different frequency regimes and provides the foundation for deriving precise pointwise and energy decay estimates.

Let $W=(\phi,u)^{\top}$ and $W_0=(\phi_0,u_0)^{\top}$. The reformulated Euler system with damping, \eqref{Main31}-\eqref{ID1}, can be expressed in compact vector form:
	\begin{equation}\label{L-system-M}
		\left\{
		\begin{aligned}
			& \partial_t W +\mathcal{A}W = F,\\
			& W(x,0) = W_0,
		\end{aligned}
		\right.
	\end{equation}
	where  $\mathcal{A}$ and $F$ are given by
	\begin{equation*}
		\mathcal{A} = \left( {\begin{array}{*{20}{c}}
				{0}&{{\rm{div}}}\\
				{\nabla}&\gamma{I}
		\end{array}} \right)
		\quad \mathrm{and}  \quad
	F=\left(
		\begin{array}{*{20}{c}}
			{-u\cdot \nabla \phi+u\cdot\nabla V}\\
			{-u\cdot \nabla u}
		\end{array} \right)
		=:\left(
		\begin{array}{*{20}{c}}
			{f_1}\\
			{f_2}
		\end{array} \right).
	\end{equation*}
Denoting by $G(x,t)$ the Green function of the linearized system, the Duhamel representation of the solution is
\[
		W(x,t)=G\ast W_0 +\int^t_0 G(t-\tau)\ast F(\tau)\mathrm{d}\tau,
\]
	where \begin{equation*}
		G(x,t) = \left(
		{\begin{array}{*{20}{c}}
				{G_{11}}&{G_{12}}\\
				{G_{21}}&{G_{22}}
		\end{array}} \right).
	\end{equation*}
Although the structure of $G$ has been studied in \cite{STW2003}, for the sake of completeness and to adapt it to our present framework, we briefly recall its key properties and provide a self-contained derivation of the spectral estimates.
	
	The Green function $G$ satisfies 
	\begin{equation*}
		\left\{
		\begin{aligned}
			& \partial_t G +\mathcal{A}G = 0,\\
			& G(x,t) = \delta(x)I,
		\end{aligned}
		\right.
	\end{equation*}
	where $\delta(x)$ is the Dirac delta distribution. Taking the Fourier transform of the above system gives
	\begin{equation*}
		\left\{
		\begin{aligned}
			& \partial_t \hat{G} +\hat{\mathcal{A}}\hat{G} = 0,\\
			& \hat{G}(x,0) = I,
		\end{aligned}
		\right.
	\end{equation*}
with
	\begin{equation*}
		\hat{\mathcal{A}} = \left( {\begin{array}{*{20}{c}}
				{0}&{i\xi^{\top}}\\
				{i\xi}&\gamma {I}
		\end{array}} \right).
	\end{equation*}
The decay properties of $G$ are governed by the eigenvalues of $\hat{\mathcal{A}}$, which satisfy
	\begin{equation*}
		|\lambda I+\hat{\mathcal{A}}|{\rm{ = }}\left| {\begin{array}{*{20}{c}}
				{\lambda}&{i\xi}\\
				{i\xi}&({\lambda+\gamma})I\\
		\end{array}} \right|
		=(\lambda+\gamma)^2(\lambda^2+\gamma \lambda+|\xi|^2)=0.
	\end{equation*}
Thus, the eigenvalues are
	\begin{equation*}
		\lambda_1=\lambda_2=-\gamma,
		\quad \lambda_3=\frac{-\gamma+\sqrt{\gamma^2-4|\xi|^2}}{2}, \quad \mathrm{and}
		\quad \lambda_4=\frac{-\gamma-\sqrt{\gamma^2-4|\xi|^2}}{2}.
	\end{equation*}
	
	The following lemma describes the behavior of the eigenvalues $\lambda_3$ and $\lambda_4$ in the low- and high-frequency regimes.
	\begin{lemma}\label{lem-tezheng}
		Let $r_0<\frac{\gamma}{2}$ be a small positive constant. 
		\begin{itemize}
			\item For  $|\xi|< r_0$, the eigenvalues admit the expansions
			\begin{equation*}
				\lambda_3=-\frac{1}{\gamma}|\xi|^2+o(|\xi|^2),
				\quad \lambda_4=-\gamma+\frac{1}{\gamma}|\xi|^2+o(|\xi|^2).
			\end{equation*}
			which shows the diffusive decay of the low-frequency mode $\lambda_3$.
			\item For $|\xi|\geq r_0$, there exists a positive constant $\eta$ such that
\[
				{\rm{Re}}\lambda_i\leq -\eta, \quad i=3, 4.
\]
		\end{itemize}
	\end{lemma}

	\begin{proof}
		For $|\xi|< r_0$, the expansions follow by a direct Taylor expansion of the square root in $\lambda_{3,4}$.
		For $|\xi|> \frac{\gamma}{2}$, the roots $\lambda_3$ and $\lambda_4$ are complex with
		\begin{equation*}
			{\rm{Re}}\lambda_3={\rm{Re}}\lambda_4=-\frac{\gamma}{2}.
		\end{equation*}
		For $r_0\leq|\xi|\leq \frac{\gamma}{2}$, the roots are real, and a straightforward estimate gives
		\begin{align*}
			\lambda_3&=\frac{-\gamma+\sqrt{\gamma^2-4|\xi|^2}}{2}
			=\frac{-2|\xi|^2}{\gamma+\sqrt{\gamma^2-4|\xi|^2}}
			\leq -\frac{1}{\gamma} |\xi|^2
			\leq -\frac{1}{\gamma} r_0^2,\\
			\lambda_4&=\frac{-\gamma-\sqrt{\gamma^2-4|\xi|^2}}{2}
			\leq -\frac{\gamma}{2}.
		\end{align*}
Thus, letting $ \eta=\min\{\frac{1}{\gamma} r_0^2, \frac{\gamma}{2}\}$ completes the proof.
	\end{proof}
	
The eigenvalue analysis in Lemma \ref{lem-tezheng} immediately translates into decay properties of the Green function for the linearized system. In particular, the low-frequency modes behave diffusively (governed by a heat kernel), while the high-frequency modes decay exponentially. This allows us to obtain the following sharp $L^2$-decay estimates for each component of the Green function and their low-frequency parts.

	\begin{lemma}\label{gl-lam}
	Let $\phi$ be a smooth function. Then for any integer $k\geq 0$, the Green function components satisfy
		\begin{align}
			&\|\nabla^k G_{11}\ast \phi\|_{L^2}
			\leq C(1+t)^{-\frac{3}{4}-\frac{k}{2}}(\|\phi\|_{L^1}
			+\|\nabla^k\phi\|_{L^2}), \label{G11}\\
			&\|\nabla^k G_{12}\ast \phi\|_{L^2}
			\leq C(1+t)^{-\frac{5}{4}-\frac{k}{2}}(\|\phi\|_{L^1}
			+\|\nabla^k\phi\|_{L^2}),  \label{G12}\\
			&\|\nabla^k G_{21}\ast \phi\|_{L^2}
			\leq C(1+t)^{-\frac{5}{4}-\frac{k}{2}}(\|\phi\|_{L^1}
			+\|\nabla^k\phi\|_{L^2}),  \label{G21}\\
			&\|\nabla^k G_{22}\ast \phi\|_{L^2}
			\leq C(1+t)^{-\frac{7}{4}-\frac{k}{2}}(\|\phi\|_{L^1}
			+\|\nabla^k\phi\|_{L^2}).\label{G22}
		\end{align}
		In particular, for the low-frequency part of $G_{ij}$, $i,j=1,2$, we have
		\begin{align}
			&\|\nabla^k G_{11}^{\ell}\ast \phi\|_{L^2}\leq C(1+t)^{-\frac{3}{4}-\frac{k}{2}}\|\phi\|_{L^1},\label{GL11}\\
			&\|\nabla^k G_{12}^{\ell}\ast \phi\|_{L^2}\leq C(1+t)^{-\frac{5}{4}-\frac{k}{2}}\|\phi\|_{L^1}, \label{GL12}\\
			&\|\nabla^k G_{21}^{\ell}\ast \phi\|_{L^2}\leq C(1+t)^{-\frac{5}{4}-\frac{k}{2}}\|\phi\|_{L^1}, \label{GL21}\\
			&\|\nabla^k G_{22}^{\ell}\ast \phi\|_{L^2}\leq C(1+t)^{-\frac{7}{4}-\frac{k}{2}}\|\phi\|_{L^1}.\label{GL22}
		\end{align}
	\end{lemma}
	\begin{proof}
By direct computation, the Fourier transform of the Green function admits the explicit representation
		\begin{align}\label{01001}
			\hat{G}(t,\xi)
			= \left( {\begin{array}{*{20}{c}}
					{\hat{G}_{11}}&{\hat{G}_{12}}\\
					{\hat{G}_{21}}&{\hat{G}_{22}}
			\end{array}} \right)
			=\left( {\begin{array}{*{20}{c}}
					{\frac{\lambda_3 e^{\lambda_4t}
							-\lambda_4e^{\lambda_3t}}{\lambda_3-\lambda_4}}
					&{-i\xi^{\top}\frac{e^{\lambda_3 t}
							-e^{\lambda_4 t}}{\lambda_3-\lambda_4}}\\
					{-i\xi\frac{e^{\lambda_3 t}
							-e^{\lambda_4t}}{\lambda_3-\lambda_4}}
					&{e^{-\gamma t}I
						+\big(\frac{\lambda_3 e^{\lambda_3 t}
							-\lambda_4e^{\lambda_4t}}{\lambda_3-\lambda_4}
						-e^{-\gamma t}\big)\frac{\xi\xi^{\top}}{|\xi|^2}}
			\end{array}} \right).
		\end{align}
For any smooth $\phi$, using Plancherel's identity, we obtain
		\begin{align*}
			\| \nabla^kG_{11}\ast \phi\|^2_{L^2}
			={}&C\|(i\xi)^{k}\hat{G}_{11}\hat{\phi}\|^2_{L^2}
			=C\left(\int_{|\xi|<r_0}
			+\int_{|\xi|\geq r_0}\right)|(i\xi)^{k}\hat{G}_{11}\hat{\phi}|^2\mathrm{d}\xi\\
			\leq{}& C\int_{|\xi|<r_0} |(i\xi)^{k}e^{-\frac{1}{\gamma}|\xi|^2t}\hat{\phi}|^2\mathrm{d}\xi
			+C\int_{|\xi|\geq r_0} |(i\xi)^{k}e^{-\eta t}\hat{\phi}|^2\mathrm{d}\xi\\
			\leq{}&C\|\widehat{\nabla^k H}\hat{\phi}\|_{L^2}^2
			+C\|e^{-\eta t}\widehat{\delta(x)}\widehat{\nabla^k \phi}\|^2_{L^2}\\
			\leq{}&C\|\nabla^k H\|_{L^2}^2\|\phi\|_{L^1}^2
			+C e^{-2\eta t}\|\nabla^k\phi\|_{L^2}^2   \\
			\leq{}&C(1+t)^{-\frac{3}{2}-k}(\|\phi\|_{L^1}^2+\|\nabla^k\phi\|_{L^2}^2),
		\end{align*}
		where $H(x,t)=(4\pi{t})^{-\frac{3}{2}}e^{-\frac{|x|^2}{4t}}$ is the heat kernel. This proves \eqref{G11}.  The estimates \eqref{G12}-\eqref{G22} follow from similar arguments applied to the corresponding Fourier multipliers.

For the low-frequency parts, using the definition of the projection in \eqref{011404} and restricting to $|\xi|<r_0$, we estimate
		\begin{align*}
			\| \nabla^kG_{11}\ast \phi\|^2_{L^2}&=C\|(i\xi)^{k}\hat{G}_{11}\hat{\phi}\|^2_{L^2}\cr
			&=C\int_{|\xi|<r_0}|(i\xi)^{k}\hat{G}_{11}\hat{\phi}|^2\mathrm{d}\xi\\
			&\leq C\int_{|\xi|<r_0} |(i\xi)^{k}e^{-\frac{1}{\gamma}|\xi|^2t}\hat{\phi}|^2\mathrm{d}\xi\\
			&\leq C\|\widehat{\nabla^k H}\hat{\phi}\|_{L^2}^2\\
			&\leq C\|\nabla^k H\|_{L^2}^2\|\phi\|_{L^1}^2\\
			&\leq C(1+t)^{-\frac{3}{2}-k}\|\phi\|_{L^1}^2,
		\end{align*}
which gives \eqref{GL11}. The others \eqref{GL12}-\eqref{GL22} follow analogously.
This completes the proof.
	\end{proof}

%
%
%
%
%
%
%
%
%
%
%
\subsubsection{Time-weighted energy method for the nonlinear system}
We now derive the large-time decay rates for the full nonlinear system, completing the proof of Theorem \ref{thm3}.  Our approach combines the low-high frequency decomposition introduced above with a time-weighted energy method, which allows us to track the decay behavior of each derivative level and control nonlinear interactions over long times.

To quantify the time-decay behavior at different derivative levels, we introduce the basic energy 
\[
			\mathcal{E}_j(t)=\|\nabla^j \phi (t)\|_{H^{3-j}}^2+\|\nabla^j u(t)\|_{H^{3-j}}^2 \quad \mbox{for } j=0,1,2,3,
\]
		and define the time-weighted energy functional
		\begin{align}\label{mskt}
			Q(t)=\sup_{0<\tau\leq t}\Big\{
			(1+\tau)^{\frac{3}{4}}\|\phi(\tau)\|_{L^2}
			+(1+\tau)^{\frac{5}{4}}\|\nabla\phi(\tau)\|_{H^2}
			&+(1+\tau)^{\frac{5}{4}}\|u(\tau)\|_{H^3}\Big\}.
		\end{align}
The goal is to prove that $Q(t)$ remains uniformly bounded in time. By definition, the decay rates we aim to establish follow immediately once $Q(t)$ is controlled:
\[
			\|\phi(t)\|_{L^2}\leq (1+t)^{-\frac{3}{4}}Q(t),\quad \|\nabla\phi(t)\|_{H^2}\leq (1+t)^{-\frac{5}{4}}Q(t),\quad \|u(t)\|_{H^3}\leq (1+t)^{-\frac{5}{4}}Q(t).
\]

	\begin{lemma}\label{l68}
		Assume that all the conditions in Theorem \ref{thm3} hold. Then there exists a positive constant $C$ independent of time such that
				\begin{align}
			\|(\phi,u)(t)\|_{H^3}\leq C(1+t)^{-\frac{3}{4}}\left(\delta_0+ N_0+\delta_0 Q(t)\right).\label{011405}
		\end{align}
	\end{lemma}
	\begin{proof}
For $\mathcal{E}_0(t)$, standard energy estimates yield
		\begin{equation*}
			\frac{d}{dt}\mathcal{E}_0(t)+C(\|\nabla\phi(t)\|_{H^2}^2+\|u(t)\|_{H^3}^2)\leq 0.
		\end{equation*}
To control the high-frequency part, we use Lemma \ref{lema6} which implies $\|\phi^{h}(t)\|_{L^2}\leq C \|\nabla\phi(t)\|_{L^2}$. Thus there exists $\tilde{C}_5 > 0$ such that
		\begin{equation}\label{lem-hs1}
			\frac{d}{dt}\mathcal{E}_0(t)+\tilde{C}_5\mathcal{E}_0(t)\leq C\|\phi^{\ell}(t)\|^2_{L^2}.
		\end{equation}
For the lower frequency part, we utilize the Duhamel representation:		
\[
			W^\ell(x,t)=G^\ell*U_0+\int_0^tG^\ell(t-\tau)*F(\tau)d\tau,
\]
where $F=(f_1,f_2)^\top$ with
		\begin{align*}
			f_1=-u\cdot \nabla \phi+u\cdot\nabla V, \quad f_2=-u\cdot \nabla u.
		\end{align*}
Using the Green function estimates from Lemma \ref{gl-lam} and H\"older's inequality, we obtain
		\begin{align*}
			&\int_0^t \big(\|G_{11}^\ell(t-\tau)*f_1(\tau)\|_{L^2} +\|G_{12}^\ell(t-\tau)*f_2(\tau)\|_{L^2}\big)\,d\tau\\
			&\quad \leq C\int_0^t(1+t-\tau)^{-\frac{3}{4}}\Big(\|(u\cdot \nabla\phi)(\tau)\|_{L^1}+\|(u\cdot\nabla V)(\tau)\|_{L^1}\Big)d\tau \\
			&\qquad +C\int_0^t(1+t-\tau)^{-\frac{5}{4}}\|(u\cdot\nabla u)(\tau)\|_{L^1}d\tau\\
			&\quad \leq C\int_0^t(1+t-\tau)^{-\frac{3}{4}}\Big(\|u\|_{L^2} \|\nabla\phi\|_{L^2}+\|u\|_{L^2}\|\nabla V\|_{L^2}\Big)d\tau \\
			&\qquad +C\int_0^t(1+t-\tau)^{-\frac{5}{4}}\|u\|_{L^2}\|\nabla u\|_{L^2}d\tau\\
			&\quad \leq C\left(\int_0^t(1+t-\tau)^{-\frac{3}{2}}\|u(\tau)\|_{L^2}^2d\tau\right)^{\frac{1}{2}}
			\left(\int_0^t\|\nabla\phi(\tau)\|_{L^2}^2d\tau\right)^{\frac{1}{2}}\\
			&\qquad +C\delta_0\int_0^t(1+t-\tau)^{-\frac{3}{4}} (1+\tau)^{-\frac{5}{4}}Q(t)d\tau\\
			&\qquad +C\left(\int_0^t(1+t-\tau)^{-\frac{5}{2}}\|u(\tau)\|_{L^2}^2d\tau\right)^{\frac{1}{2}}\left(\int_0^t\|\nabla u(\tau)\|_{L^2}^2d\tau\right)^{\frac{1}{2}}\\
			&\quad \leq C(1+t)^{-\frac{3}{4}}\delta_0Q(t),
		\end{align*}
where in the last step we used the energy dissipation bound
		\begin{equation}\label{D1}
			\int_0^t(\|\nabla \phi(\tau)\|_{L^2}^2+\|\nabla u(\tau)\|_{L^2}^2)d\tau\leq 
			C\delta_0^2.
		\end{equation}
		This yields
		\begin{align}\label{lem-hs3}
		\begin{aligned}
			\|\phi^{\ell}(t)\|_{L^2}
			&\leq \|G_{11}^\ell*\phi_0+G_{12}^\ell*u_0\|_{L^2}\cr
			&\quad +\int_0^t\|G_{11}^\ell(t-\tau)*f_1(\tau)+G_{12}^\ell(t-\tau)*f_2(\tau)\|_{L^2}d\tau\\
			&\leq C(1+t)^{-\frac{3}{4}}(\|\phi_0\|_{L^1}+\|u_0\|_{L^1})
			+C(1+t)^{-\frac{3}{4}}\delta_0Q(t)\\
			&\leq C(1+t)^{-\frac{3}{4}}( N_0+\delta_0Q(t)).
			\end{aligned}
		\end{align}
		Substituting \eqref{lem-hs3} into \eqref{lem-hs1}, we obtain
		\begin{equation*}
			\frac{d}{dt}\mathcal{E}_0(t)+\tilde{C}_5\mathcal{E}_0(t)\leq C(1+t)^{-\frac{3}{2}}\left( N_0+\delta_0Q(t)\right)^2.
		\end{equation*}
		Applying Gr\"{o}nwall's inequality gives
		\begin{align*}
			\mathcal{E}_0(t)=\|(\phi,u)(t)\|_{H^3}^2
			\leq C(1+t)^{-\frac{3}{2}}\left(\delta_0+ N_0+\delta_0 Q(t)\right)^2,
		\end{align*}
which proves \eqref{011405}.
\end{proof}		
		
At this stage, we have obtained the desired $t^{-\frac{3}{4}}$ decay for the basic energy $\mathcal{E}_0(t)$. However, the decay rates for higher-order derivatives of $\phi$ and $u$ are not yet optimal. In the next step, we refine the estimate by exploiting the additional regularity encoded in $\nabla\phi$ and $u$ to achieve the sharper $t^{-\frac{5}{4}}$ decay rates.

\begin{lemma}\label{l69}  Under the assumptions of Theorem \ref{thm3}, the higher-order derivatives of $(\phi,u)$ satisfy
		\begin{equation}\label{011406}
			\|(\nabla \phi,\nabla u)(t)\|_{H^2}\leq C(1+t)^{-\frac{5}{4}}	(\delta_0+N_0+\delta_0Q(t))
		\end{equation}
for a positive constant $C$ independent of time.
\end{lemma}		
\begin{proof}	
\noindent \emph{Step 1: Higher-order energy estimates.} Following the approach in the proof of \eqref{A3} and using the smallness of $\delta_0$, for $k \in {1,2,3}$ we find
		\begin{align}\label{011204}
			&\frac{1}{2}\frac{d}{dt}(\|\nabla^k \phi \|_{L^2}^2+\|\nabla^k u\|_{L^2}^2)
			+\gamma \|\nabla^ku\|_{L^2}^2 \leq C\delta_0(\|\nabla^k \phi \|_{L^2}^2+\|\nabla^ku\|_{L^2}^2+\|\nabla u\|_{H^1}^2).
		\end{align}
This captures the dissipation of $u$ and couples it with $\phi$ at derivative level $k$. However, to exploit the smoothing effect from the elliptic operator $\mathcal{K}_\infty$ on the high-frequency part of $\phi$, we refine the estimate below.

		Let $k=1,2,3$. Applying the operator $\nabla^{k-1} \mathcal{K}_{\infty}$ to \eqref{Main31}, multiplying the resulting equation by $\nabla^{k} \phi^h$, and integrating over $\mathbb{R}^3$, using Lemma \ref{lema6}, we obtain
		\begin{align}\label{0102}
		\begin{aligned}
			&\|\nabla^k \phi^h\|_{L^2}^2+\frac{d}{dt}\int_{\mathbb{R}^3} \nabla^{k-1}u^h\cdot\nabla^k \phi^hdx\\
			&\quad =-\int_{\mathbb{R}^3}\nabla^{k-1}u^h\cdot\nabla^k\text{div}u^hdx-\gamma\int_{\mathbb{R}^3} \nabla^{k-1}u^h\cdot\nabla^k\phi^hdx\\
			&\qquad -\int_{\mathbb{R}^3}\nabla^{k-1}\mathcal{K}_\infty(u\cdot\nabla V)\cdot\nabla^k \phi ^hdx  -\int_{\mathbb{R}^3}\nabla^k\mathcal{K}_\infty(u\cdot\nabla \phi)\cdot \nabla^{k-1}u^hdx \cr
			&\qquad -\int_{\mathbb{R}^3}\nabla^{k-1}\mathcal{K}_\infty(u\cdot\nabla u)\cdot \nabla^k \phi ^hdx.
			\end{aligned}
		\end{align}
		The first term and the second term of the right-hand side of \eqref{0102} can be estimated as follows
		\begin{align*}
			&\Big|\int_{\mathbb{R}^3}\nabla^{k-1}u^h\cdot\nabla^k\text{div}u^hdx\Big|+\gamma\Big|\int_{\mathbb{R}^3}\nabla^{k-1}u^h\cdot\nabla^k \phi ^hdx\Big|\\
			&\quad \leq \|\nabla^ku^h\|_{L^2}^2+\gamma\|\nabla^{k-1}u^h\|_{L^2}\|\nabla^{k}\phi^h\|_{L^2}\\
			&\quad \leq C\|\nabla^ku\|_{L^2}^2+C\|\nabla^{k}u\|_{L^2}\|\nabla^{k}\phi^h\|_{L^2}\\
			&\quad \leq C\|\nabla^ku\|_{L^2}^2+\frac{1}{2}\|\nabla^{k}\phi^h\|_{L^2}^2.
		\end{align*}
		For the third term, we divide 
		\begin{align*}
			&\int_{\mathbb{R}^3}\nabla^{k-1}\mathcal{K}_\infty(u\cdot\nabla V)\cdot\nabla^k \phi ^hdx\\
			&\quad =\int_{\mathbb{R}^3}\nabla^{k-1}(u\cdot\nabla V)\cdot\nabla^k \phi ^hdx
			-\int_{\mathbb{R}^3}\nabla^{k-1}\mathcal{K}_1(u\cdot\nabla V)\cdot\nabla^k \phi ^hdx,
		\end{align*}
where
		\begin{align*}
			&\int_{\mathbb{R}^3}\nabla^{k-1}(u\cdot\nabla V)\cdot\nabla^k \phi ^hdx\\
			&\quad =\int_{\mathbb{R}^3}(\nabla^{k-1}(u\cdot\nabla V)-u\cdot\nabla^kV)\cdot\nabla^k\phi^hdx
			+\int_{\mathbb{R}^3}u\cdot\nabla^kV\cdot\nabla^k\phi^hdx\\
			&\quad \leq C(\|\nabla u\|_{L^\infty}\|\nabla^{k-1}V\|_{L^2}+\|\nabla^{k-1}u\|_{L^6}\|\nabla V\|_{L^3})\|\nabla^k\phi^h\|_{L^2}\\
			&\qquad +C\|u\|_{L^6}\|\nabla^kV\|_{L^3}\|\nabla^k\phi^h\|_{L^2}\\
			&\quad \leq C\delta_0\|\nabla^2u\|_{H^1}\|\nabla^k\phi^h\|_{L^2}
			+C\delta_0\|\nabla^ku\|_{L^2}\|\nabla^k\phi^h\|_{L^2}+C\delta_0\|\nabla u\|_{L^2}\|\nabla^k\phi^h\|_{L^2}\\
			&\quad \leq C\delta_0(\|\nabla^k\phi^h\|_{L^2}^2+\|\nabla^2u\|_{H^1}^2+
			\|\nabla^ku\|_{L^2}^2)
		\end{align*}	
and
		\begin{align*}
			\int_{\mathbb{R}^3}\nabla^{k-1}\mathcal{K}_1(u\cdot\nabla V)\cdot\nabla^k\phi^hdx 
			&\leq C\|u\cdot\nabla V\|_{L^2}\|\nabla^k\phi^h\|_{L^2}\\
			&\leq C\|u\|_{L^6}\|\nabla V\|_{L^3}\|\nabla^k\phi^h\|_{L^2}\\
			&\leq C\delta_0 \|\nabla u\|_{L^2}\|\nabla^k\phi^h\|_{L^2}\\
			&\leq C\delta_0(\|\nabla^k\phi^h\|_{L^2}^2+\|\nabla u\|_{L^2}^2).
		\end{align*}
This gives
		\begin{align*}
			\int_{\mathbb{R}^3}\nabla^{k-1}\mathcal{K}_\infty(u\cdot\nabla V)\cdot\nabla^k \phi ^hdx \leq C\delta_0(\|\nabla^k\phi^h\|_{L^2}^2+\|\nabla u\|_{H^2}^2+\|\nabla^ku\|_{L^2}^2).
		\end{align*}

For the fourth term, we observe that 	for $k\geq 2$ 
		\begin{align}\label{011301}
		\begin{aligned}
			&	\Big|\int_{\mathbb{R}^3}\nabla^k(u\cdot\nabla \phi)\cdot \nabla^{k-1}udx\Big|\\
			&\quad \leq \Big|\int_{\mathbb{R}^3}\nabla^{k-1}(u\cdot\nabla \phi)\cdot \nabla^{k-1}\text{div}udx\Big|\\
			&\quad \leq \Big|\int_{\mathbb{R}^3}(\nabla^{k-1}(u\cdot\nabla \phi)-u\cdot\nabla^{k-1}\nabla \phi)\cdot \nabla^{k-1}\text{div}udx\Big|\cr
			&\qquad +\Big|\int_{\mathbb{R}^3}(u\cdot\nabla^{k-1}\nabla \phi)\cdot \nabla^{k-1}\text{div}udx\Big|\\
			&\quad \leq \big\|[\nabla^{k-1},u]\nabla\phi\big\|_{L^2}\|\nabla^ku\|_{L^2}+C\|u\|_{L^\infty}\|\nabla^k \phi \|_{L^2}\|\nabla^ku\|_{L^2}\\
			&\quad \leq C(\|\nabla u\|_{L^3}\|\nabla^{k-1}\phi\|_{L^6}+\|\nabla^{k-1}u\|_{L^6}\|\nabla \phi\|_{L^3})\|\nabla^{k}u\|_{L^2}+C\delta_0\|\nabla^k \phi \|_{L^2}\|\nabla^ku\|_{L^2}\\
			&\quad \leq C\delta_0(\|\nabla^k \phi \|_{L^2}^2+\|\nabla^{k}u\|_{L^2}^2).
			\end{aligned}
		\end{align}
		We also readily find 
		\begin{equation}\label{011302}
			\begin{aligned}
				&	\Big|\int_{\mathbb{R}^3}\nabla (u\cdot\nabla \phi)\cdot udx\Big|\leq C\|u\|_{L^\infty}
				\|\nabla\phi\|_{L^2}\|\text{div} u\|_{L^2}\leq C\delta_0(\|\nabla\phi\|_{L^2}^2+\|\nabla u\|_{L^2}^2).
			\end{aligned}
		\end{equation}
Thus, combining \eqref{011301} and \eqref{011302} gives 
		\begin{align*}
			\Big|\int_{\mathbb{R}^3}\nabla^k(u\cdot\nabla \phi)\cdot \nabla^{k-1}udx\Big| \leq C\delta_0(\|\nabla^k \phi \|_{L^2}^2+\|\nabla^{k}u\|_{L^2}^2)
		\end{align*}
		for $k\geq 1$. Hence, by splitting $\mathcal{K}_\infty=I-\mathcal{K}_1$ we obtain
		\begin{align*}
			&\Big|\int_{\mathbb{R}^3}\nabla^k\mathcal{K}_\infty(u\cdot\nabla \phi)\cdot \nabla^{k-1}udx\Big|\nonumber\\
			&\quad \leq \Big|\int_{\mathbb{R}^3}\nabla^k(u\cdot\nabla \phi)\cdot \nabla^{k-1}udx\Big|+	\Big|\int_{\mathbb{R}^3}\nabla^k\mathcal{K}_1(u\cdot\nabla \phi)\cdot \nabla^{k-1}u dx\Big|\nonumber\\
			&\quad \leq C\delta_0(\|\nabla^k \phi \|_{L^2}^2+\|\nabla^{k}u\|_{L^2}^2).
		\end{align*}
		In a similar manner, we also get
		\begin{align*}
			\Big|\int_{\mathbb{R}^3}\nabla^{k-1}\mathcal{K}_\infty(u\cdot\nabla u)\cdot \nabla^k \phi ^hdx\Big|
			\leq C\delta_0(\|\nabla^k \phi \|_{L^2}^2+\|\nabla^{k}u\|_{L^2}^2).
		\end{align*}
		Combining all the above estimates, we arrive at
		\begin{align*}
			&\|\nabla^k \phi^h\|_{L^2}^2+\frac{d}{dt}\int_{\mathbb{R}^3} \nabla^{k-1}u^h\cdot\nabla^k \phi^hdx\nonumber\\
			&\quad \leq C\|\nabla^ku\|_{L^2}^2+C\delta_0(\|\nabla^k \phi \|_{L^2}^2+\|\nabla^{k}u\|_{L^2}^2+\|\nabla u\|_{H^2}^2)
			+\frac{1}{2}\|\nabla^k\phi^h\|_{L^2}^2, 
		\end{align*}
that is,
		\begin{align}\label{011205}
		\begin{aligned}
			&\frac{1}{2}\|\nabla^k \phi^h\|_{L^2}^2+\frac{d}{dt}\int_{\mathbb{R}^3} \nabla^{k-1}u^h\cdot\nabla^k \phi^hdx\\
			&\quad \leq C\|\nabla^ku\|_{L^2}^2+C\delta_0(\|\nabla^k \phi \|_{L^2}^2+\|\nabla^{k}u\|_{L^2}^2+\|\nabla u\|_{H^2}^2).
			\end{aligned}
		\end{align}
To exploit the damping in \eqref{011205}, we introduce a small parameter $\kappa_2>0$ and combine \eqref{011204} with $\kappa_2\times\eqref{011205}$:		
		\begin{align*}
			&	\frac{d}{dt}\left(\frac{1}{2}\big(\|\nabla^k \phi \|_{L^2}^2+\|\nabla^k u\|_{L^2}^2\big)+ \kappa _2\int_{\mathbb{R}^3} \nabla^{k-1}u^h\cdot\nabla^k \phi^hdx\right)\\
			&\quad+\gamma\|\nabla^ku\|_{L^2}^2+\frac{1}{2} \kappa _2\|\nabla^k \phi^h\|_{L^2}^2\nonumber\\
			&\qquad \leq  C\delta_0(\|\nabla^k \phi \|_{L^2}^2+\|\nabla^ku\|_{L^2}^2+\|\nabla u\|_{H^1}^2)
			+C\delta_0\kappa_2(\|\nabla^k \phi \|_{L^2}^2+\|\nabla^{k}u\|_{L^2}^2+\|\nabla u\|_{H^2}^2).
		\end{align*}
Note that by Lemma \ref{lema6} and the smanllness of $ \kappa _2$ 	
\[
				\frac{1}{2}(\|\nabla^k \phi\|_{L^2}^2+\|\nabla^k u\|_{L^2}^2)+ \kappa _2\int_{\mathbb{R}^3} \nabla^{k-1}u^h\cdot\nabla^k \phi^hdx \sim (\|\nabla^k \phi\|_{L^2}^2+\|\nabla^k u\|_{L^2}^2),
\]	
thus choosing $\kappa_2$ and $\delta_0$ sufficiently small, we conclude 				
		\begin{align}\label{011304}
		\begin{aligned}
			&\frac{d}{dt}(\|\nabla^k \phi\|_{L^2}^2+\|\nabla^k u\|_{L^2}^2)
			+C(\gamma\|\nabla^ku\|_{L^2}^2+\kappa_2\|\nabla^k\phi^h\|_{L^2}^2)\\
			&\quad \leq C\delta_0\|\nabla^k \phi\|_{L^2}^2+C\delta_0\|\nabla u\|_{H^1}^2+C\delta_0\kappa _2\|\nabla u\|_{H^2}^2.
			\end{aligned}
		\end{align}
Summing \eqref{011304} over $k=1,2,3$ and using the smallness of $\delta_0$ and $\kappa_2$, we obtain for some   $\tilde{C}_6>0$,
		\begin{align}
			\frac{d}{dt}(\|\nabla \phi\|_{H^2}^2+\|\nabla u\|_{H^2}^2)
			+\tilde{C}_6(\|\nabla \phi\|_{H^2}^2+\|\nabla u\|_{H^2}^2)
			\leq C\|\nabla\phi^\ell\|_{L^2}^2.\label{0113010}
		\end{align}
	
\medskip
\noindent \emph{Step 2: Low-frequency decay and closure via Gr\"onwall's inequality.} We finally estimate the $L^2$ time decay of $\nabla\phi^\ell$.  Using the Duhamel representation for the low-frequency part, we estimate
		\begin{equation}\label{0113011}
			\begin{aligned}
				&\|\nabla\phi^{\ell}(t)\|_{L^2}\\
				&\quad \leq C(1+t)^{-\frac{5}{4}}\|W_0\|_{L^1}\cr
				&\qquad +C\int_0^t(1+t-\tau)^{-\frac{5}{4}}
				\big(\|(u\cdot\nabla\phi)(\tau)\|_{L^1}+\|(u\cdot\nabla u)(\tau)\|_{L^1}  +\|(u\cdot\nabla V)(\tau)\|_{L^1}\big)d\tau\\
				&\quad \leq C(1+t)^{-\frac{5}{4}}N_0 \cr
				&\qquad + C\int_0^t(1+t-\tau)^{-\frac{5}{4}}\big(\|u\|_{L^2} \|\nabla\phi\|_{L^2}+\|u\|_{L^2}\|\nabla u\|_{L^2}+\|u\|_{L^2}\|\nabla V\|_{L^2}\big)d\tau\\
				&\quad \leq C(1+t)^{-\frac{5}{4}}N_0+C\delta_0\int_0^t(1+t-\tau)^{-\frac{5}{4}} (1+\tau)^{-\frac{5}{4}}Q(t)d\tau\\
				&\qquad +C\left(\int_0^t(1+t-\tau)^{-\frac{5}{2}}\|u(\tau)\|_{L^2}^2d\tau\right)^{\frac{1}{2}}
				\left(\int_0^t(\|\nabla\phi(\tau)\|_{L^2}^2+\|\nabla u(\tau)\|_{L^2}^2)d\tau\right)^{\frac{1}{2}} \\
				&\quad \leq C(1+t)^{-\frac{5}{4}}(N_0+\delta_0Q(t)).
			\end{aligned}
		\end{equation}
		Then we deduce from \eqref{0113010} and \eqref{0113011} that 
		\begin{equation*}
			\frac{d}{dt}\mathcal{E}_{1}(t)+\tilde{C}_6\mathcal{E}_{1}(t)\leq C(1+t)^{-\frac{5}{2}}
			(N_0+\delta_0Q(t))^2,
		\end{equation*}
		and applying Gr\"{o}nwall's inequality yields
		\begin{equation*}
			\mathcal{E}_1(t)=\|(\nabla\phi,\nabla u)(t)\|_{H^2}^2\leq C(1+t)^{-\frac{5}{2}}(\delta_0+N_0+\delta_0Q(t))^2.
		\end{equation*}
		This completes the proof.
 \end{proof}
 %
 %
 %
 %
 %
 %
 %
 %
 %
  \subsubsection{Refined decay estimates for higher-order derivatives}

We now refine the decay analysis to obtain improved rates for higher-order derivatives. 
Building on Lemma \ref{l68}, which provides preliminary decay control, we incorporate the refined derivative bounds derived in the previous subsection into the time-weighted framework. 
This allows us to close the estimates for the functional $Q(t)$ introduced in \eqref{mskt}, yielding uniform-in-time bounds and thereby establishing enhanced large-time decay rates for $\phi$ and $u$ up to third-order derivatives. 

\begin{proposition}\label{p68}  Under the assumptions of Theorem \ref{thm3}, there exists a constant $C>0$ independent of time such that 
		\begin{align*}
			\|\phi(t)\|_{L^2}&\leq C(\delta_0+ N_0)(1+t)^{-\frac{3}{4}},\\
			\|\nabla \phi(t)\|_{H^2}&\leq C(\delta_0+ N_0)(1+t)^{-\frac{5}{4}},\\
			\|u(t)\|_{H^3}&\leq C(\delta_0+ N_0)(1+t)^{-\frac{5}{4}}.	
		\end{align*}
\end{proposition}	
\begin{proof}			
We first improve the decay rates of $\|u(t)\|_{L^2}$. From $\eqref{Main31}_2$, we obtain the Duhamel representation
		\begin{align}
			u=e^{-\gamma t}u_0-\int_0^te^{-\gamma(t-\tau)}(\nabla\phi+u\cdot\nabla u)(\tau)d\tau.
			\label{011408}
		\end{align}
Using the bounds obtained in \eqref{011406} for $\nabla \phi$ and $\nabla u$, it follows that
		\begin{align}
			\|u\|_{L^2}&\leq e^{-\gamma t}\|u_0\|_{L^2}+\int_0^te^{-\gamma(t-\tau)}\|(\nabla\phi+u\cdot\nabla u)(\tau)\|_{L^2}d\tau\nonumber\\
			&\leq e^{-\gamma t}\|u_0\|_{L^2}+C\int_0^te^{-\gamma(t-\tau)}(\|\nabla\phi\|_{L^2}+\|\nabla u\|_{H^1}\|\nabla u\|_{L^2})d\tau\nonumber\\
			&\leq C(1+t)^{-\frac{5}{4}}(\delta_0+N_0+\delta_0 Q(t)).\label{011407}
		\end{align}
		
Hence, combining \eqref{011406}, \eqref{011407}, and \eqref{011405} with the definition $Q(t)$ in \eqref{mskt}, we obtain
		\begin{align*}
			Q(t)\leq C(\delta_0+N_0+\delta_0 Q(t)).
		\end{align*}
Choosing $\delta_0$ sufficiently small gives 
		\begin{align*}
			Q(t)\leq C(\delta_0+N_0),
		\end{align*}
and substituting this back into the definition of $Q(t)$ yields
		\begin{align*}
			\|\phi(t)\|_{L^2}&\leq C(\delta_0+ N_0)(1+t)^{-\frac{3}{4}},\\
			\|\nabla \phi(t)\|_{H^2}&\leq C(\delta_0+ N_0)(1+t)^{-\frac{5}{4}},\\
			\|u(t)\|_{H^3}&\leq C(\delta_0+ N_0)(1+t)^{-\frac{5}{4}}.
		\end{align*}
		This completes the proof.
	\end{proof}
	
Noting that the decay rates in Proposition \ref{p68} for higher-order derivatives remain relatively slow, we now turn to improving these rates. In particular, we derive sharper decay estimates for the second derivatives of $\phi$ and the gradient of $\nabla u$, which will be crucial for the long-time regularity analysis.
	\begin{proposition}\label{p69}
		Under the assumptions of Theorem \ref{thm3}, there exists a constant $C>0$ independent of time such that 
\[
			\|\nabla^2\phi(t)\|_{H^1}\leq C(\delta_0+N_0+N_0^2)(1+t)^{-\frac{7}{4}}, \quad \|\nabla u(t)\|_{H^2}\leq C(\delta_0+N_0+N_0^2)(1+t)^{-\frac{7}{4}}.
\]
	\end{proposition}
	\begin{proof}
		Define a new time-weighted functional $M(t)$ as 
		\begin{align}\label{011306}
			M(t)=\sup_{0<\tau\leq t}\Big\{
			(1+\tau)^{\frac{7}{4}}\|\nabla^2\phi(\tau)\|_{H^1}
			+(1+\tau)^{\frac{7}{4}}\|\nabla u(\tau)\|_{H^2}
			\Big\}.
		\end{align}
It is immediate that
		\begin{align*}
			\|\nabla^2\phi(t)\|_{H^1}\leq (1+t)^{-\frac{7}{4}}M(t),\quad \|\nabla u(t)\|_{H^2}\leq (1+t)^{-\frac{7}{4}}M(t).
		\end{align*}
Summing \eqref{011304} for $k=2,3$, we obtain
		\begin{align}\label{011305}
		\begin{aligned}
			&\frac{d}{dt}(\|\nabla^2\phi\|_{H^1}^2+\|\nabla^2 u\|_{H^1}^2)
			+C(\|\nabla^2\phi\|_{H^1}^2+\|\nabla^2u\|_{H^1}^2)\\
			&\quad \leq C\|\nabla^2\phi^\ell\|_{L^2}^2+C(\delta_0+\delta_0\kappa _2)\|\nabla u\|_{L^2}^2\\
			&\quad \leq C\|\nabla^2\phi^\ell\|_{L^2}^2+C\delta_0\|\nabla u\|_{L^2}^2.
			\end{aligned}
		\end{align}
		For the estimates on the right-hand side of the above, using the Duhamel representation and Green function bounds for the low-frequency part, we derive
			\begin{align*}
				\|\nabla^2\phi^{\ell}(t)\|_{L^2}
				&\leq C(1+t)^{-\frac{7}{4}}\|W_0\|_{L^1}\cr
				&\quad 	+C\int_0^t(1+t-\tau)^{-\frac{7}{4}}
				\big(\|(u\cdot\nabla\phi)(\tau)\|_{L^1}+\|(u\cdot\nabla u)(\tau)\|_{L^1}\big)d\tau\\
				&\quad+C\int_0^\frac{t}{2}(1+t-\tau)^{-\frac{7}{4}} \|u\cdot\nabla V\|_{L^1}d\tau
				+C\int_\frac{t}{2}^t(1+t-\tau)^{-\frac{3}{2}} \|u\cdot\nabla V\|_{L^2}d\tau\\
				&\quad+C\int_\frac{t}{2}^t(1+t-\tau)^{-\frac{5}{4}} \|\nabla u\cdot\nabla V\|_{L^1}d\tau\\
				&\leq C(1+t)^{-\frac{7}{4}}N_0+ C\int_0^t(1+t-\tau)^{-\frac{7}{4}}\big(\|u\|_{L^2} \|\nabla\phi\|_{L^2}+\|u\|_{L^2}\|\nabla u\|_{L^2}\big)d\tau\\
				&\quad+C\int_0^{\frac{t}{2}} (1+t-\tau)^{-\frac{7}{4}}\|u\|_{L^2}\|\nabla V\|_{L^2}d\tau +C\int_{\frac{t}{2}}^t(1+t-\tau)^{-\frac{3}{2}}\|u\|_{L^6}\|\nabla V\|_{L^3}d\tau\\
				&\quad+C\int_{\frac{t}{2}}^t(1+t-\tau)^{-\frac{5}{4}}\|\nabla u\|_{L^2}\|\nabla V\|_{L^2}d\tau\\
				&\leq C(1+t)^{-\frac{7}{4}}N_0+ C(\delta_0+N_0)^2\int_0^t(1+t-\tau)^{-\frac{7}{4}}(1+\tau)^{-\frac{5}{2}}d\tau\\
				&\quad+C\delta_0(\delta_0+N_0)\int_0^{\frac{t}{2}} (1+t-\tau)^{-\frac{7}{4}}(1+\tau)^{-\frac{5}{4}}d\tau\cr
				&\quad +C\delta_0M(t)\int_{\frac{t}{2}}^t(1+t-\tau)^{-\frac{3}{2}}(1+\tau)^{-\frac{7}{4}}d\tau\\
				&\quad+C\delta_0M(t)\int_{\frac{t}{2}}^t(1+t-\tau)^{-\frac{5}{4}}(1+\tau)^{-\frac{7}{4}}d\tau\\
				&\leq C(1+t)^{-\frac{7}{4}}(N_0+(\delta_0+N_0)^2+\delta_0(\delta_0+N_0)+\delta_0M(t))\\
				&\leq C(1+t)^{-\frac{7}{4}}(\delta_0+N_0+N_0^2+\delta_0M(t)).
			\end{align*}
Additionally, 
			\begin{align*}
			\|\nabla u\|_{L^2}^2\leq (1+t)^{-\frac{7}{2}}M(t)^2.	
			\end{align*}	
Substituting these bounds into \eqref{011305} gives
		\begin{align*}
			&\frac{d}{dt}(\|\nabla^2\phi\|_{H^1}^2+\|\nabla^2 u\|_{H^1}^2)
			+C(\|\nabla^2\phi\|_{H^1}^2+\|\nabla^2u\|_{H^1}^2)\nonumber\\
			&\quad \leq C(1+t)^{-\frac{7}{2}}((\delta_0+N_0+N_0^2+\delta_0M(t))^2+\delta_0M(t)^2).
		\end{align*}
Applying Gr\"{o}nwall's inequality, we deduce
		\begin{align*}
		\begin{aligned}
			&\|\nabla^2\phi(t)\|_{H^1}
			+\|\nabla^2u(t)\|_{H^1}  \leq C(1+t)^{-\frac{7}{4}}(\delta_0+N_0+N_0^2+\delta_0M(t)+\delta_0^{\frac{1}{2}}M(t)).
			\end{aligned}
		\end{align*}
Using the representation \eqref{011408} for $u$ and similar decay arguments, we also obtain  
		\begin{align*}
		\|\nabla u(t)\|_{H^2}\leq C(1+t)^{-\frac{7}{4}}(\delta_0+N_0+N_0^2+\delta_0M(t)+\delta_0^{\frac{1}{2}}M(t)).
		\end{align*}
Substituting these into the definition \eqref{011306} yields
		\begin{align*}
			M(t)\leq C(\delta_0+N_0+N_0^2+\delta_0M(t)+\delta_0^{\frac{1}{2}}M(t)),
		\end{align*}
and for $\delta_0$ small enough, we have 
		\begin{align*}
			M(t)\leq C(\delta_0+N_0+N_0^2).
		\end{align*}
		Therefore, we conclude
		\begin{align*}
			\|\nabla^2\phi(t)\|_{H^1}\leq C(\delta_0+N_0+N_0^2)(1+t)^{-\frac{7}{4}},\quad \|\nabla u(t)\|_{H^2}\leq C(\delta_0+N_0+N_0^2)(1+t)^{-\frac{7}{4}}.
		\end{align*}
		This completes the proof.
	\end{proof}

%
%
%
%
%
%
%
%
%
%
%
%
%
%
\subsubsection{Proof of Theorem \ref{thm3}: large-time behavior}

Combining the refined decay estimates established in Propositions \ref{p68} and \ref{p69}, we obtain the following decay rates for the solution:
	\begin{align*}
	\|\nabla^k\phi(t)\|_{L^2}&\leq C(1+t)^{-\frac{3}{4}-\frac{k}{2}},\quad k=0,1,2,\nonumber\\
	\|\nabla^ku(t)\|_{L^2}&\leq  C(1+t)^{-\frac{5}{4}-\frac{k}{2}},\quad k=0,1,\nonumber\\
	\|\nabla^3\phi(t)\|_{L^2}&\leq C(1+t)^{-\frac{7}{4}},\nonumber\\
	\quad \|\nabla ^2u(t)\|_{H^1}&\leq C(1+t)^{-\frac{7}{4}},
\end{align*}
where the positive constant $C>0$ is independent of time and depends on $\delta_0$ and $N_0$. These decay estimates, together with the global well-posedness of the Cauchy problem \eqref{Main31}-\eqref{ID1}, complete the proof of Theorem \ref{thm3}.
	
	We emphasize that these bounds provide sharp upper decay rates for the nonlinear system; the matching lower bounds, which confirm their optimality, will be established in the next section.

%
%
%
%
%
%
%
%
%
%
%
%
\section{Optimal decay estimates of solutions}\label{sec:opt-decay} 
In this section, we complete the proof of Theorem \ref{thm4} by establishing \emph{matching lower bounds} for the decay of solutions to the Euler system with damping. 
%
%
%
%
%
%
%
%
%
%
%
%
\subsection{Optimal decay for the linearized system}
We first revisit the linearized Euler system with damping and establish sharp decay estimates for its solutions. While the upper bounds for $\phi$ and $u$ obtained in Theorem \ref{thm3} already exhibit the optimal decay rates, the corresponding lower bounds require a more delicate analysis. In particular, we exploit the non-degeneracy condition on the initial density perturbation \eqref{in-data-optimal} to ensure that the leading diffusive mode persists at large times. Using the explicit Green function representation of the linearized problem and a careful low-frequency analysis in the Fourier space, we derive matching upper and lower bounds for the $L^2$-decay of $\phi$ and $u$, thereby capturing the precise asymptotic behavior of the linear system.

	\begin{proposition}\label{pro-li-d}
		Assume that all the conditions in Theorem \ref{thm4} hold. Let $(\phi,u)$ be the global classical solution to the linear system \eqref{L-system-M}. Then there exists $t_0>0$ such that for all $t\geq t_0$,
		\begin{align}\label{011411}
		\begin{aligned}
			d_*(1+t)^{-\frac{3}{4}}&\leq \|\phi(t)\|_{L^2}\leq C(1+t)^{-\frac{3}{4}},\\
			d_*(1+t)^{-\frac{5}{4}}&\leq \|u(t)\|_{L^2}\leq C(1+t)^{-\frac{5}{4}}, 
			\end{aligned}
		\end{align}
		where $d_*$ and  $C$ are positive constants independent of time.
	\end{proposition}
	\begin{proof}
We work in the Fourier space and use the explicit Green function representation in \eqref{01001}. For $\hat{\phi}(\xi,t)$, decompose the solution into three parts:
		\begin{align}\label{011410}
		\begin{aligned}
			\hat{\phi}(\xi,t)&={\frac{\lambda_3 e^{\lambda_4t}
					-\lambda_4e^{\lambda_3t}}{\lambda_3-\lambda_4}}\hat{\phi}_0-\frac{e^{\lambda_3 t}
				-e^{\lambda_4 t}}{\lambda_3-\lambda_4}i\xi\cdot\hat{u}_0\\
			&=-\frac{\lambda_4e^{\lambda_3t}}{\lambda_3-\lambda_4}\hat{\phi}_0+\frac{\lambda_3e^{\lambda_4t}}{\lambda_3-\lambda_4}\hat{\phi}_0-\frac{e^{\lambda_3t}-e^{\lambda_4t}}{\lambda_3-\lambda_4}i\xi\cdot\hat{u}_0.
			\end{aligned}
		\end{align}
		
The leading contribution comes from the first term (low-frequency diffusion mode). On $|\xi|<r_0$, we apply Lemma \ref{lem-tezheng} to approximate $\lambda_3\sim -\frac{1}{\gamma}|\xi|^2$ and $\lambda_4\sim -\gamma$. Thus, we obtain		
		\begin{align*}
		\begin{aligned}
			\left\|\frac{\lambda_4e^{\lambda_3t}}{\lambda_3-\lambda_4}\hat{\phi}_0\right\|_{L^2}^2
			&=\int_{\mathbb{R}^3}\frac{\lambda_4^2e^{2\lambda_3t}}{|\lambda_3-\lambda_4|^2}|\hat{\phi}_0|^2d\xi\\
			&\geq \int_{|\xi|\leq r_0}\frac{\lambda_4^2e^{2\lambda_3t}}{|\lambda_3-\lambda_4|^2}|\hat{\phi}_0|^2d\xi \\
			&\geq Cc_0^2 \int_{\mathbb{R}^3}\frac{\lambda_4^2e^{2\lambda_3t}}{|\lambda_3-\lambda_4|^2}d\xi \\
			&\geq Cc_0^2(1+t)^{-\frac{3}{2}},
			\end{aligned}
		\end{align*}
where $c_0>0$ comes from the non-degeneracy assumption on $\hat{\phi}_0$ in \eqref{in-data-optimal}.		
		
The remaining terms in \eqref{011410} decay faster. Specifically,
		\begin{align*}
			\left\|\frac{\lambda_3e^{\lambda_4t}}{\lambda_3-\lambda_4}\hat{\phi}_0-\frac{e^{\lambda_3t}-e^{\lambda_4t}}{\lambda_3-\lambda_4}i\xi\cdot\hat{u}_0\right\|_{L^2}
			&\leq 
			Ce^{-\eta t }(\|\phi_0\|_{L^2}+\|u_0\|_{L^2})+
			C\left\||\xi|e^{-\frac{1}{\gamma}|\xi|^2t}\hat{u}_0\right\|_{L^2}\\
			&\leq C(1+t)^{-\frac{5}{4}}(\|\phi_0\|_{L^2}+\|u_0\|_{L^2}+\|u_0\|_{L^1})\\
			&\leq C(1+t)^{-\frac{5}{4}},
		\end{align*}
Combining these estimates, for $t\geq t_0$ we obtain
		\begin{align*}
			\|\phi(t)\|_{L^2}&\geq 	Cc_0(1+t)^{-\frac{3}{4}}-C(1+t)^{-\frac{5}{4}}\\
			&\geq 	Cc_0(1+t)^{-\frac{3}{4}}-C(1+t_0)^{-\frac{1}{2}}(1+t)^{-\frac{3}{4}}\\
			&\geq \frac{1}{2}Cc_0(1+t)^{-\frac{3}{4}}.
		\end{align*}
		This establishes the sharp decay rate for $\phi$.
		
		For $\hat{u}(t,\xi)$, using \eqref{01001}, we decompose:
		\begin{align*}
		\begin{aligned}
			\hat{u}(t,\xi)&=-\frac{e^{\lambda_3 t}
				-e^{\lambda_4 t}}{\lambda_3-\lambda_4}i\xi^{\top}\hat{\phi}_0+({e^{-\gamma t}I
				+\big(\frac{\lambda_3 e^{\lambda_3 t}
					-\lambda_4e^{\lambda_4t}}{\lambda_3-\lambda_4}
				-e^{-\gamma t}\big)\frac{\xi\xi^{\top}}{|\xi|^2}})\hat{u}_0\\
			&=-\frac{e^{\lambda_3 t}
			}{\lambda_3-\lambda_4}i\xi^{\top}\hat{\phi}_0+\frac{
				e^{\lambda_4 t}}{\lambda_3-\lambda_4}i\xi^{\top}\hat{\phi}_0\cr
				&\quad 
			+\left({e^{-\gamma t}I
				+\big(\frac{\lambda_3 e^{\lambda_3 t}
					-\lambda_4e^{\lambda_4t}}{\lambda_3-\lambda_4}
				-e^{-\gamma t}\big)\frac{\xi\xi^{\top}}{|\xi|^2}}\right)\hat{u}_0.
				\end{aligned}
		\end{align*}
The first term dominates at low frequency. On $|\xi|<r_0$, we deduce
		\begin{align*}
		\begin{aligned}
			\left\|\frac{e^{\lambda_3 t}
			}{\lambda_3-\lambda_4}i\xi^{\top}\hat{\phi}_0\right\|_{L^2}^2&\geq
			\int_{|\xi|<r_0}\frac{e^{2\lambda_3t}}{|\lambda_3-\lambda_4|^2}|\xi|^2|\hat{\phi}_0|^2d\xi
 \geq Cc_0^2(1+t)^{-\frac{5}{2}}.
			\end{aligned}
		\end{align*}
The remaining parts decay faster:
		\begin{align*}
		\begin{aligned}
			&\left\|\frac{
				e^{\lambda_4 t}}{\lambda_3-\lambda_4}i\xi^{\top}\hat{\phi}_0+({e^{-\gamma t}I
				+\big(\frac{\lambda_3 e^{\lambda_3 t}
					-\lambda_4e^{\lambda_4t}}{\lambda_3-\lambda_4}
				-e^{-\gamma t}I\big)\frac{\xi\xi^{\top}}{|\xi|^2}})\hat{u}_0\right\|_{L^2}\\
			&\quad \leq 
			Ce^{-\eta t }(\|\phi_0\|_{L^2}+\|u_0\|_{L^2})+Ce^{-\gamma t}\|u_0\|_{L^2}+
			C\left\||\xi|^2e^{-\frac{1}{\gamma}|\xi|^2t}\hat{u}_0\right\|_{L^2}\\
			&\quad \leq C(1+t)^{-\frac{7}{4}}(\|\phi_0\|_{L^2}+\|u_0\|_{L^2}+\|u_0\|_{L^1})\\
			&\quad \leq C(1+t)^{-\frac{7}{4}}.
			\end{aligned}
		\end{align*}
Thus, we have
		\begin{align*}		
			\|u(t)\|_{L^2} &=	\|\hat{u}(t)\|_{L^2} \geq 	Cc_0(1+t)^{-\frac{5}{4}}-C(1+t)^{-\frac{7}{4}} \geq \frac{1}{2}Cc_0(1+t)^{-\frac{5}{4}}.
		\end{align*}
Letting $d_*=\frac{1}{2}Cc_0$ completes the proof of \eqref{011411}.
	\end{proof}
 
	%
	%
	%
	%
	%
	%
	%
	%
	%

	\subsection{Extension to the nonlinear system and proof of Theorem \ref{thm4}}
We now extend the sharp decay estimates of the linearized system to the full nonlinear Euler system with damping, thus completing the proof of Theorem \ref{thm4}. Using Duhamel's principle, the solution is decomposed into the linear evolution of the initial data and a time-convolution of the nonlinear source terms. While the nonlinear contributions decay at a comparable rate to the linear part, their coefficients can be made sufficiently small due to the smallness of the initial data. This ensures that the leading diffusive mode of the linearized solution dominates the long-time behavior. By combining the sharp decay of the linearized system with refined convolution estimates for the nonlinear terms, we derive matching upper and lower bounds for the $L^2$-decay of $\phi$ and $u$, as well as for their derivatives.

Our starting point is the Duhamel representation of the solution:
	\begin{align*}
		&\phi(x,t)=G_{11}*\phi_0+G_{12}*u_0+\int_0^t (G_{11}(t-\tau)*f_1(\tau)+G_{12}(t-\tau)*f_2(\tau))d\tau,\\
		&u(x,t)=G_{21}*\phi_0+G_{22}*u_0+	\int_0^t (G_{21}(t-\tau)*f_1(\tau)+G_{22}(t-\tau)*f_2(\tau))d\tau,
	\end{align*}
where $G_{ij}$ denotes the Green's functions of the linearized system, and $f_1,f_2$ are the nonlinear forcing terms. 
To obtain sharp lower bounds for $\phi$ and $u$, it suffices to control the Duhamel terms and show that they decay strictly faster than the leading linear contributions.
	
Note that
	\begin{align*}
		&\int_0^t \|G_{11}(t-\tau)*f_1(\tau)+G_{12}(t-\tau)*f_2(\tau)\|_{L^2}d\tau \\
		&\quad \leq C\int_0^t(1+t-\tau)^{-\frac{3}{4}}(\|(u\cdot \nabla\phi)(\tau)\|_{L^1}+\|(u\cdot \nabla V)(\tau)\|_{L^1}+\|(u\cdot\nabla u)(\tau)\|_{L^1})d\tau\\
		&\qquad+C\int_0^te^{-\eta(t-\tau)}(\|(u\cdot \nabla\phi)(\tau)\|_{L^2}+\|(u\cdot \nabla V)(\tau)\|_{L^2}+\|(u\cdot\nabla u)(\tau)\|_{L^2})d\tau.
	\end{align*}
By Proposition \ref{p69}, we obtain
	\begin{align}\label{N001}
	\begin{aligned}
		&\|(u\cdot \nabla\phi)(\tau)\|_{L^1}+\|(u\cdot \nabla V)(\tau)\|_{L^1}+\|(u\cdot\nabla u)(\tau)\|_{L^1}\\
		&\quad \leq \|u\|_{L^2}\|\nabla\phi\|_{L^2}++\|u\|_{L^2}\|\nabla V\|_{L^2}+\|u\|_{L^2}\|\nabla u\|_{L^2}\\
		&\quad \leq C(1+\tau)^{-\frac{5}{4}}(\delta_0+ N_0)\|\nabla(\phi,u)\|_{L^2}
		+C(1+\tau)^{-\frac{5}{4}}\delta_0(\delta_0+ N_0),
		\end{aligned}
	\end{align}
	and 
	\begin{align}\label{N002}
	\begin{aligned}
		&\|(u\cdot \nabla\phi)(\tau)\|_{L^2}+\|(u\cdot \nabla V)(\tau)\|_{L^2}+\|(u\cdot\nabla u)(\tau)\|_{L^2}\\
		&\quad \leq \|u\|_{L^\infty}\|\nabla\phi\|_{L^2}+\|u\|_{L^\infty}\|\nabla V\|_{L^2}+\|u\|_{L^\infty}\|\nabla u\|_{L^2}\\
		&\quad \leq C\|\nabla u\|_{H^1}\|\nabla\phi\|_{L^2}+C\|\nabla u\|_{H^1}\|\nabla V\|_{L^2}+C\|\nabla u\|_{H^1}\|\nabla u\|_{L^2}\\
		&\quad \leq C(1+\tau)^{-\frac{7}{4}}(\delta_0+ N_0+N_0^2)\|\nabla(\phi,u)\|_{L^2}
		+C(1+\tau)^{-\frac{7}{4}}\delta_0(\delta_0+ N_0+N_0^2).
		\end{aligned}
	\end{align}
It then follows from \eqref{N001}, \eqref{N002}, and \eqref{D1} that
	\begin{align*}
		&\int_0^t \|G_{11}(t-\tau)*f_1(\tau)+G_{12}(t-\tau)*f_2(\tau)\|_{L^2}d\tau\nonumber\\
		&\quad \leq C(\delta_0+ N_0)\int_0^t(1+t-\tau)^{-\frac{3}{4}}(1+\tau)^{-\frac{5}{4}}(\|\nabla (\phi,u)(\tau)\|_{L^2}+\delta_0)d\tau\nonumber\\
		&\qquad +C(\delta_0+ N_0+N_0^2)\int_0^te^{-\eta(t-\tau)}(1+\tau)^{-\frac{7}{4}}(\|\nabla (\phi,u)(\tau)\|_{L^2}+\delta_0)d\tau\nonumber\\
		&\quad \leq C(\delta_0+ N_0+N_0^2)\left(\int_0^t\big((1+t-\tau)^{-\frac{3}{2}}(1+\tau)^{-\frac{5}{2}}+e^{-2\eta(t-\tau)}(1+\tau)^{-\frac{5}{2}}\big)d\tau\right)^{\frac{1}{2}}\nonumber\\
		&\hspace{6cm} \times \Big(\int_0^t(\|\nabla (\phi,u)(\tau)\|_{L^2}^2)d\tau\Big)^{\frac{1}{2}} \cr
		&\qquad +C\delta_0(\delta_0+N_0+N_0^2)(1+t)^{-\frac{3}{4}}\nonumber\\
		&\quad \leq C(1+t)^{-\frac{3}{4}}(\delta_0^2+\delta_0 N_0+\delta_0N_0^2). 
	\end{align*}
	Thus, by Parseval's identity, Proposition \ref{pro-li-d}, we deduce for sufficiently small $\delta_0$,  
	\begin{align}\label{thm3-p1}
		\begin{aligned}
		\|\phi(t)\|_{L^2}
		&\geq \|G_{11}*\phi_0+G_{12}*u_0\|_{L^2} \cr
		&\quad -\int_0^t \|G_{11}(t-\tau)*f_1(\tau)+G_{12}(t-\tau)*f_2(\tau)\|_{L^2}d\tau\\
		&\geq d_*(1+t)^{-\frac{3}{4}}- C(1+t)^{-\frac{3}{4}}
		(\delta_0^2+\delta_0 N_0+\delta_0N_0^2)\\
		&\geq \frac{1}{2}d_*(1+t)^{-\frac{3}{4}}.
				\end{aligned}
	\end{align}
A similar argument gives 
	\begin{align*}
		\|u(t)\|_{L^2}\geq 	\frac{1}{2}d_*(1+t)^{-\frac{5}{4}}.
	\end{align*}
	
Finally, we extend these bounds to higher-order derivatives using the interpolation inequality. We observe for $t$ large enough
	\begin{align}\label{thm3-p4}
	\begin{aligned}
		\|\Lambda^{-1}\phi(t)\|_{L^2}&\leq \|\Lambda^{-1}\phi^\ell(t)\|_{L^2}
		+\|\Lambda^{-1}\phi^h(t)\|_{L^2}\\
		&\leq C(1+t)^{-\frac{1}{4}}+C\int_0^t(1+t-\tau)^{-\frac{1}{4}}\|\nabla (\phi,u)(\tau)\|_{L^1}d\tau+C\|\phi^h(t)\|_{L^2}\\
		&\leq C(1+t)^{-\frac{1}{4}}+C\int_0^t(1+t-\tau)^{-\frac{1}{4}}(1+\tau)^{-\frac{5}{4}}d\tau+C\|\phi(t)\|_{L^2}\\
		&\leq C(1+t)^{-\frac{1}{4}}.
		\end{aligned}
	\end{align}
	It then follows from Lemma \ref{lema5} that 
\[
		\|\phi\|_{L^2}\leq C\|\Lambda^{-1}\phi\|_{L^2}^{\frac{k}{k+1}}
		\|\nabla^k \phi\|_{L^2}^{\frac{1}{k+1}},
\]
	which, together with \eqref{thm3-p1} and \eqref{thm3-p4}, yields for $k=0,1,2$ that 
\[
		\|\nabla^k\phi(t)\|_{L^2}=\|\nabla^k\phi(t)\|_{L^2}\geq d_0(1+t)^{-\frac{3}{4}-\frac{k}{2}},
\]
for some $d_0 > 0$ independent of $t$.
	
	Similarly, there exists a positive constant $\bar{d}_0$ such that 
	\begin{align*}
		\|\nabla^k u(t)\|_{L^2}\geq \bar{d}_0(1+t)^{-\frac{5}{4}-\frac{k}{2}},\quad k=0,1,
	\end{align*}
	which completes the proof of Theorem \ref{thm4}.

%
%
%
%
%
%
%
%
%
%
%
%
%
\section{Blow-up analysis for the pressureless damped Euler system} \label{Sec:thm2}
In this section, we analyze the pressureless Euler system with damping, focusing on finite-time singularity formation. Specifically, we consider
\begin{equation}\label{Main111}
	\left\{
	\begin{aligned}	
		&\partial_t\rho+\textrm{div}(\rho u)=0,\\
		&\partial_t(\rho u)+\text{div}(\rho u\otimes u)+\gamma\rho u=0,
	\end{aligned}
	\right.
\end{equation}
with initial data and far-field condition
\begin{align}\label{BC11}
	(\rho,u)(x,0)=(\rho_0,u_0),\quad \text{and}\quad  \lim _{|x| \rightarrow+\infty}(\rho, u)(x,t)=\left(1, 0\right).
\end{align}

Compared to the isothermal system, \eqref{Main111} lacks a pressure term, so the dispersive and smoothing effects of pressure are absent. Although damping provides a stabilizing mechanism, it acts only on momentum and cannot fully counterbalance nonlinear compression. This structural difference necessitates a separate analysis: we show that even for small initial data, bounded damping may fail to prevent singularity formation, leading to finite-time blow-up.

Before turning to the blow-up analysis, we recall the local existence of classical solutions, which follows from standard energy estimates for symmetrizable hyperbolic systems (see \cite{HKK2014}).
 \begin{theorem}\label{locthm}
Assume that the initial data satisfy $(\rho_0-1,u_0)\in H^3(\mathbb{R}^3) \times H^4(\R^3)$. Then there exists a short time $T_0>0$ such that the system  \eqref{Main111}-\eqref{BC11} admits a unique classical solution $(\phi,u)$ satisfying
		\begin{align*}
			&\rho - 1 \in C^0([0,T_0];H^3(\mathbb{R}^3))\cap C^1([0,T_0];H^2(\mathbb{R}^3)),\\
			&u\in C^0([0,T_0];H^3(\mathbb{R}^4))\cap C^1([0,T_0];H^3(\mathbb{R}^3)).
		\end{align*}
\end{theorem}

\subsection{Proof of Theorem \ref{thm2}} We now turn to the proof of Theorem \ref{thm2}. The key idea is to construct a set of \emph{weighted energy-momentum functionals} that amplify the compressive modes of the flow and lead to a contradiction with their positivity. This approach builds on the virial-type method but is adapted to incorporate the effects of linear damping.

\medskip
\noindent \emph{Step 1: Weighted functionals and differential inequality.}  To study global behavior, we recall the weighted momentum and mass-like quantities:
	\begin{align*}
	 A_1(t)=\int_{\mathbb{R}^3}\rho u\cdot\nabla H dx,\quad A_2(t)=\int_{\mathbb{R}^3}\rho Hdx>0,\quad H(x)=\gamma e^{-\frac{|x|^2}{\gamma}}.
	\end{align*}
Differentiating in time and using the equations of motion, we derive  
	\begin{align*}
		\frac{dA_2(t)}{dt}=\int_{\mathbb{R}^3}\partial_t\rho H\,dx= \int_{\mathbb{R}^3}\rho u\cdot\nabla H\,dx =A_1(t),
	\end{align*}
	and 
	\begin{align*}
		\frac{dA_1(t)}{dt}&=\int_{\mathbb{R}^3}\partial_t(\rho u)\cdot\nabla H dx\\
		&=-\int_{\mathbb{R}^3}(\text{div}(\rho u\otimes u)+\gamma \rho u)\cdot\nabla Hdx\\
		&=\int_{\mathbb{R}^3}\rho u\otimes u:D^2Hdx-\gamma A_1(t). 	
	\end{align*}
Note that
	\begin{align*}
		\nabla H=-2xe^{-\frac{|x|^2}{\gamma}},\quad \Delta H=-6e^{-\frac{|x|^2}{\gamma}}+\frac{4|x|^2}{\gamma}e^{-\frac{|x|^2}{\gamma}},	\end{align*}
	which implies that 
	\begin{align*}
		|D^2H|^2\leq 10e^{-\frac{|x|^2}{\gamma}}\leq \frac{10H}{\gamma}.
	\end{align*}
Define the particle trajectory $X(s;x,t)$ passing though $(x,t)$ by
		\begin{equation*}
			\left\{
			\begin{aligned}
				&\partial_s X(s;x,t) = u(X(s;x,t), s),\\
				&X(t;x,t)=x.
			\end{aligned}
			\right.
		\end{equation*}
		Then along this particle trajectory, we obtain 
		\begin{equation*}
		\partial_s u(X(s;x,t), s) = -\gamma u(X(s;x,t), s),
		\end{equation*}
		and thus 
		\begin{equation*}
			 u(x,t)=u_0(X(0;x,t))e^{-\gamma t}.
		\end{equation*}
In particular, we have
		\begin{align*}
			\|u(t)\|_{L^\infty}\leq \|u_0\|_{L^\infty}\leq C_*\|u_0\|_{H^3}
		\end{align*}
		for some $C_* > 0$ independent of $t$. This yields
		\begin{align*}
			\frac{dA_1(t)}{dt}+\gamma A_1(t)&=\int_{\mathbb{R}^3}\rho u\otimes u: D^2Hdx\\
			&\leq \|u\|_{L^\infty}\left(\int_{\mathbb{R}^3}\rho |u|^2dx\right)^{\frac{1}{2}}\left(\int_{\mathbb{R}^3}\rho |D^2H|^2dx\right)^{\frac{1}{2}}\\	
			&\leq \frac{C_*\|u_0\|_{H^3}}{\gamma^{\frac{1}{2}}}\left(\int_{\mathbb{R}^3}\rho |u|^2dx+\int_{\mathbb{R}^3}\rho Hdx\right)\\
			&\leq \frac{C_*a_0}{\gamma^{\frac{1}{2}}}\left(\int_{\mathbb{R}^3}\rho |u|^2dx+\int_{\mathbb{R}^3}\rho Hdx\right)\\
			&\leq D_* \left(\int_{\mathbb{R}^3}\rho |u|^2dx+\int_{\mathbb{R}^3}\rho Hdx\right),
		\end{align*}
		where $D_*=\frac{C_*a_0}{\gamma^{\frac{1}{2}}}$.   Here we remind that $\rho(t) \notin L^1(\R^3)$.

	The basic energy estimate gives
	\begin{align*}
		\frac{1}{2}\frac{d}{dt}\int_{\mathbb{R}^3}\rho|u|^2dx+\gamma\int_{\mathbb{R}^3}\rho |u|^2dx=0,
	\end{align*}
and thus
	\begin{align*}
		\int_{\mathbb{R}^3}\rho|u|^2dx\leq \int_{\mathbb{R}^3}\rho_0|u_0|^2dx=E_0.
	\end{align*}
	Combining these estimates together, we obtain a closed second-order differential inequality for $A_2(t)$:
	\begin{align*}
		\frac{d^2A_2(t)}{dt^2}+\gamma \frac{dA_2(t)}{dt}\leq D_*A_2(t)+D_*E_0.	
	\end{align*}
	Choosing $c_1=\gamma,~c_2=D_*,~c_3=D_*E_0$, according to Lemma \ref{lem2}, we deduce from the above inequality that 
	\begin{align*}
		A_2(t)&\leq\left(A_{2}(0)+\frac{c_3}{\beta(\beta +c_1)}+\frac{1}{\gamma + 2\beta}\left(A_1(0)-\beta A_2(0)-\frac{c_3}{\beta +\gamma}\right)\right)e^{\beta t}\\
		&\quad -\frac{1}{\gamma+ 2\beta}\left(A_1(0)-\beta A_2(0)-\frac{c_3}{\beta +\gamma }\right)e^{-(\gamma+\beta)t}-\frac{c_3}{\beta(\beta+\gamma)},
	\end{align*}
	where $\beta>0$ is given by
	\begin{align*}
		\beta=\frac{-\gamma+\sqrt{\gamma^2 + 4c_2}}{2}.
	\end{align*}

\medskip
\noindent \emph{Step 2: Blow-up criterion.}	
	Our goal is to ensure that $A_2(t)$ becomes negative in finite time, contradicting the positivity of $\rho$. For this, it suffices to require
	\begin{align}\label{A01}
		A_2(0)+\frac{c_3}{\beta(\beta +c_1)}+\frac{1}{\gamma + 2\beta}\left(A_1(0)-\beta A_2(0)-\frac{c_3}{\beta +\gamma}\right)<0.
	\end{align}
Rearranging, this is equivalent to
	\begin{align*}
		D_*(A_2(0)+E_0)+\beta A_1(0)<0.
	\end{align*}
Since $A_1(0)<0$, this implies
	\begin{align*}
		\beta>-\frac{D_*(A_2(0)+E_0)}{A_1(0)} =:F_0.
	\end{align*}
If
	\begin{align*}
		M_*:=\frac{-A_1(0)}{A_2(0)+E_0}>\left(8C_*^2a_0^2 \right)^{\frac{1}{5}} \quad \mbox{and} \quad 	4	\left(\frac{C_*a_0}{M_*^2} \right)^2<\gamma <\frac{1}{2}M_*,
	\end{align*}
then the above conditions ensure \eqref{A01}. Thus, $A_2(t)$ becomes negative in finite time, which is impossible for a strictly positive density, forcing finite-time blow-up of solutions. This completes the proof of Theorem \ref{thm2}.

%
%
%
%
%
%
%
%

\bigbreak\noindent
\noindent \textbf{Acknowledgments.~} The work of the first author was supported by National Research Foundation of Korea (NRF) grant funded by the Korea government (MSIP) (No. 2022R1A2C1002820 and RS-2024-00406821). The work of the second author was supported by National Natural Science Foundation of China (Grant No. 12501293) and Anhui Provincial Natural Science Foundation (Grant No. 2408085QA031). The work of the third author was  supported  by China Scholarship Council(Grant No.202406880015)  and National Natural Science Foundation of China (Grant No.12001033).

\appendix

\section{Proof of the high-order derivative estimates in Proposition \ref{p31}}\label{app:high-order}

In this appendix, we provide the full details of \emph{Step 2} in the proof of Proposition \ref{p31}, 
which establishes the high-order energy estimates for $(\phi,u)$. 

We first recall the reformulated system:
\begin{equation}\label{phi-u-system}
 	\left\{
 	\begin{aligned}	
 		&\partial_t\phi+\text{div}u=-u\cdot\nabla \phi+u\cdot\nabla V,\\
 		&\partial_tu+\nabla \phi+\gamma u=-u\cdot\nabla u.
 	\end{aligned}
 	\right.
 \end{equation}
Apply $\nabla^k$ ($k=1,2,3$) to \eqref{phi-u-system} and take $L^2$ inner products with $\nabla^k\phi$ and $\nabla^k u$ to deduce
	\begin{align*}
		&\frac{1}{2}\frac{d}{dt}\|\nabla^k(\phi, u)\|_{L^2}^2+\gamma\|\nabla^k u\|_{L^2}^2\\
		&\quad =\int_{\mathbb{R}^3}\nabla^k(-u\cdot\nabla\phi+ u\cdot\nabla V)\cdot\nabla^k\phi \,dx
		+\int_{\mathbb{R}^3}\nabla^k(-u\cdot\nabla u)\cdot \nabla^k u\,dx.
	\end{align*}	
Note that
	\begin{align*}
		&\int_{\mathbb{R}^3}\nabla^k(-u\cdot\nabla\phi+ u\cdot\nabla V)\cdot\nabla^k\phi \,dx\nonumber\\
		&\quad =-\int_{\mathbb{R}^3}(\nabla^k(u\cdot\nabla\phi)-u\cdot\nabla^k\nabla\phi)\cdot\nabla^k\phi \,dx-\int_{\mathbb{R}^3}u\cdot\nabla^k\nabla\phi\cdot\nabla^k\phi \,dx\nonumber\\
		&\qquad +\int_{\mathbb{R}^3}(\nabla^k(u\cdot\nabla V)-u\cdot\nabla^k\nabla V)\cdot\nabla^k\phi \,dx+\int_{\mathbb{R}^3}u\cdot\nabla^k\nabla V\cdot\nabla^k\phi \,dx.
	\end{align*}
	By commutator estimates given in Lemma \ref{lema2}, we obtain
	\begin{align*}
		&-\int_{\mathbb{R}^3}(\nabla^k(u\cdot\nabla\phi)-u\cdot\nabla^k\nabla\phi)\cdot\nabla^k\phi \, dx-\int_{\mathbb{R}^3}u\cdot\nabla^k\nabla\phi\cdot\nabla^k\phi \,dx\nonumber\\
		&\quad \leq \|\nabla^k(u\cdot\nabla\phi)-u\cdot\nabla^k\nabla\phi\|_{L^2}\|\nabla^k\phi\|_{L^2}
		+C\|\text{div}u\|_{L^\infty}\|\nabla^k\phi\|_{L^2}^2\nonumber\\
		&\quad \leq C(\|\nabla u\|_{L^\infty}\|\nabla^k\phi\|_{L^2}+\|\nabla^ku\|_{L^2}\|\nabla\phi\|_{L^\infty})\|\nabla^k\phi\|_{L^2}+C\|\nabla^2u\|_{H^1}\|\nabla^k\phi\|_{L^2}^2\nonumber\\
		&\quad \leq C(\|\nabla^2 u\|_{H^1}\|\nabla^k\phi\|_{L^2}+\|\nabla^ku\|_{L^2}\|\nabla^2\phi\|_{H^1})\|\nabla^k\phi\|_{L^2}+C\|\nabla^2u\|_{H^1}\|\nabla^k\phi\|_{L^2}^2.
	\end{align*}	
	Summing the above inequality over $k=1,2,3$ gives
	\begin{align}\label{021701}
	\begin{aligned}
		&\sum_{k=1}^3\left(-\int_{\mathbb{R}^3}(\nabla^k(u\cdot\nabla\phi)-u\cdot\nabla^k\nabla\phi)\cdot\nabla^k\phi \,dx-\int_{\mathbb{R}^3}u\cdot\nabla^k\nabla\phi\cdot\nabla^k\phi \,dx\right)\\
		&\quad \leq C \|\nabla u\|_{H^2}\|\nabla \phi\|_{H^2}^2 \leq C\varepsilon \gamma\|\nabla u\|_{H^2}^2+C \varepsilon^{-1} \gamma^{-1}\|\nabla \phi\|_{H^2}^4
		\end{aligned}
	\end{align}
  for any $\varepsilon > 0$. 
	Similarly, we also find for $k=1,2,3$ that
	\begin{align*}
		&\int_{\mathbb{R}^3}(\nabla^k(u\cdot\nabla V)-u\cdot\nabla^k\nabla V)\cdot\nabla^k\phi \,dx+\int_{\mathbb{R}^3}u\cdot\nabla^k\nabla V\cdot\nabla^k\phi \,dx\nonumber\\
		&\quad \leq C(\|\nabla u\|_{L^\infty}\|\nabla^kV\|_{L^2}+\|\nabla^ku\|_{L^2}\|\nabla V\|_{L^\infty})\|\nabla^k\phi\|_{L^2}+C\|u\|_{L^\infty}\|\nabla^{k+1}V\|_{L^2}\|\nabla^k\phi\|_{L^2}\nonumber\\
		&\quad \leq C(\|\nabla^2 u\|_{H^1}\|\nabla^kV\|_{L^2}+\|\nabla^ku\|_{L^2}\|\nabla^2 V\|_{H^1})\|\nabla^k\phi\|_{L^2}+C\|\nabla u\|_{H^1}\|\nabla^{k+1}V\|_{L^2}\|\nabla^k\phi\|_{L^2},
	\end{align*}
and thus
	\begin{align*}
		&\sum_{k=1}^3\left(\int_{\mathbb{R}^3}(\nabla^k(u\cdot\nabla V)-u\cdot\nabla^k\nabla V)\cdot\nabla^k\phi \,dx+\int_{\mathbb{R}^3}u\cdot\nabla^k\nabla V\cdot\nabla^k\phi \,dx\right)\nonumber\\
		&\quad \leq C(\|\nabla^2u\|_{H^1}\|\nabla V\|_{H^2}+\|\nabla u\|_{H^2}\|\nabla^2V\|_{H^1})\|\nabla\phi\|_{H^2}+C\|\nabla u\|_{H^1}\|\nabla^2V\|_{H^2}\|\nabla\phi\|_{H^2}\nonumber\\
		&\quad \leq C\varepsilon\gamma\|\nabla u\|_{H^2}^2+C \varepsilon^{-1} \gamma^{-1}\epsilon_0^2\|\nabla\phi\|_{H^2}^2.
	\end{align*}
	This together with \eqref{021701} yields 
	\begin{align*}
		&\sum_{k=1}^3 \int_{\mathbb{R}^3}\nabla^k(-u\cdot\nabla\phi+ u\cdot\nabla V)\cdot\nabla^k\phi \,dx   \cr
		&\quad \leq C\varepsilon \gamma\|\nabla u\|_{H^2}^2+C \varepsilon^{-1} \gamma^{-1}\|\nabla \phi\|_{H^2}^4+C \varepsilon^{-1} \gamma^{-1}\epsilon_0^2\|\nabla\phi\|_{H^2}^2.
	\end{align*}
We also have
	\begin{align*}
		\sum_{k=1}^3 \int_{\mathbb{R}^3}\nabla^k(-u\cdot\nabla u)\cdot \nabla^k u\,dx \leq 
		C\|\nabla^2u\|_{H^1}\|\nabla u\|_{H^2}^2.
	\end{align*}
Combining the above estimates gives
	\begin{align*}
		&\frac{1}{2}\frac{d}{dt}\|\nabla(\phi, u)\|_{H^2}^2+\gamma(1-C\gamma^{-1}\|\nabla u\|_{H^2})\| \nabla u\|_{H^2}^2\nonumber\\
		&\quad \leq C\varepsilon \gamma\|\nabla u\|_{H^2}^2+C \varepsilon^{-1} \gamma^{-1}\|\nabla \phi\|_{H^2}^4+C \varepsilon^{-1}  \gamma^{-1}\epsilon_0^2\|\nabla\phi\|_{H^2}^2.
	\end{align*}	
 Choosing $\varepsilon$ small enough, we arrive at 
	\begin{align}\label{022501}
		\frac{d}{dt}\|\nabla(\phi, u)\|_{H^2}^2+\gamma(1-C\gamma^{-1}\|\nabla u\|_{H^2})\|\nabla u\|_{H^2}^2 \leq C\gamma^{-1}\|\nabla \phi\|_{H^2}^4+C\gamma^{-1}\epsilon_0^2\|\nabla\phi\|_{H^2}^2.
	\end{align}

To control $\|\nabla\phi\|_{H^2}^2$, we use a hypo-coercivity estimate. Taking $L^2$ energy estimate on \eqref{Main31} with $\nabla \phi$, we get
\begin{align*}
	&\gamma^{-1}\frac{d}{dt}\int_{\R^3} u \cdot \nabla\phi\,dx +\gamma^{-1}\|\nabla\phi\|_{L^2}^2\nonumber\\
	&\quad =\gamma^{-1} \int_{\R^3} u \cdot \nabla(-\text{div}u+f_1)\,dx - \int_{\R^3} u\cdot \nabla\phi\,dx - \gamma^{-1} \int_{\R^3} (u\cdot\nabla u) \cdot \nabla\phi\, dx \\
	&\quad \leq \gamma^{-1}\|\text{div}u\|_{L^2}^2+\gamma^{-1} \int_{\R^3} u \cdot \nabla f_1\,dx+ \|u\|_{L^2}\|\nabla\phi\|_{L^2}
	+\gamma^{-1}\|u\|_{L^6}\|\nabla u\|_{L^2}\|\nabla\phi\|_{L^3}\nonumber\\
	&\quad \leq \frac{1}{2}\gamma^{-1}\|\nabla \phi\|_{L^2}^2+C\gamma\|u\|_{H^2}^2+C\gamma^{-1}\|\nabla u\|_{H^2}^2(1+\gamma^{-2}\|\nabla \phi\|_{H^2}^2)+C\epsilon_0\gamma^{-1}\|\nabla u\|_{H^2}^2,
\end{align*}
where we used the fact that
\begin{align*}
	\gamma^{-1} \int_{\R^3} u \cdot \nabla f_1\,dx &=-\gamma^{-1} \int_{\R^3} \text{div}u \cdot f_1\,dx\cr
	& =-	\gamma^{-1} \int_{\R^3} \text{div}u \cdot (-u\cdot\nabla\phi+u\cdot\nabla V)\,dx \\	
	&\leq 	\gamma^{-1}\|\text{div}u\|_{L^2}(\|u\|_{L^6}\|\nabla\phi\|_{L^3}+\|u\|_{L^6}\|\nabla  V \|_{L^3})\nonumber\\
	&\leq C	\gamma^{-1}\|\nabla u\|_{L^2}(\|\nabla u\|_{L^2}\|\nabla \phi\|_{H^1}+\|\nabla u\|_{L^2}\|\nabla V\|_{H^1})\nonumber\\
	&\leq C	\gamma^{-1}\|\nabla u\|_{L^2}^2\|\nabla\phi\|_{H^1}+C\epsilon_0\gamma^{-1}\|\nabla u\|_{L^2}^2\nonumber\\
	&\leq C\gamma \|u\|_{H^2}^2+C\gamma^{-3}\|\nabla u\|_{H^2}^2\|\nabla\phi\|_{H^2}^2+C\epsilon_0\gamma^{-1}\|\nabla u\|_{H^2}^2.
\end{align*}
This deduces
\begin{align*}
	&\gamma^{-1}\frac{d}{dt}\int_{\R^3}u \cdot \nabla\phi\,dx +\frac{1}{2}\gamma^{-1}\|\nabla\phi\|_{L^2}^2\nonumber\\
	&\quad \leq C\gamma\|u\|_{H^2}^2+C\gamma^{-1}\|\nabla u\|_{H^2}^2(1+\gamma^{-2}\|\nabla \phi\|_{H^2}^2)+C\epsilon_0\gamma^{-1}\|\nabla u\|_{H^2}^2.
\end{align*}
Similarly, we obtain that for $k=1,2$
\begin{align*}
	&\gamma^{-1}\frac{d}{dt} \int_{\R^3} \nabla^k u \cdot \nabla^{k+1}\phi\,dx +\frac{1}{2}\gamma^{-1}\|\nabla^{k+1}\phi\|_{L^2}^2\nonumber\\
	&\quad \leq C\gamma\|u\|_{H^2}^2+C\gamma^{-1}\|\nabla u\|_{H^2}^2(1+\gamma^{-2}\|\nabla \phi\|_{H^2}^2)+C\epsilon_0\gamma^{-1}\|\nabla u\|_{H^2}^2.
\end{align*}
Combining the above gives
\begin{align}
	&\gamma^{-1}\frac{d}{dt} \sum_{k=0}^2 \int_{\R^3} \nabla^k u \cdot \nabla^{k+1}\phi\,dx+\frac{1}{2}\gamma^{-1}\|\nabla \phi\|_{H^2}^2\nonumber\\
	&\quad \leq C\gamma\|u\|_{H^2}^2+C\gamma^{-1}\|\nabla u\|_{H^2}^2(1+\gamma^{-2}\|\nabla \phi\|_{H^2}^2)+C\epsilon_0\gamma^{-1}\|\nabla u\|_{H^2}^2.\label{030605}
\end{align}
Thanks to Proposition \ref{A2}, we obtain
\begin{align*}
	\|u\|_{L^2}=\|\frac{m}{\rho}\|_{L^2}\leq \frac{2}{\rho_1}\|m\|_{L^2}\leq C\|m\|_{L^2},\quad \|m\|_{L^2}\leq \|\rho\|_{L^\infty}\|u\|_{L^2}\leq C\|u\|_{L^2},
\end{align*}
i.e., $\|u\|_{L^2}\sim \|m\|_{L^2}$. For the first-order derivatives, we estimate
\begin{align*}
	\|\nabla u\|_{L^2}&=\left\|\frac{\nabla m}{\rho}-\frac{m\cdot\nabla\rho}{\rho^2}\right\|_{L^2}\\
	&\leq \frac{2}{\rho_1}\|\nabla m\|_{L^2}+\frac{4}{\rho_1^2}\| m\|_{L^6}\|\nabla\rho\|_{L^3}\nonumber\\
	&\leq C\|\nabla m\|_{L^2}+C\|\nabla m\|_{L^2}(\|\nabla(\rho-\rho_\infty)\|_{H^1}+\|\nabla \rho_\infty\|_{H^1})\nonumber\\
	&\leq C\|\nabla m\|_{L^2}+C (B_0^4+B_0^2+\epsilon_0) \|\nabla m\|_{L^2}\nonumber\\
	&\leq C\|\nabla m\|_{L^2},
\end{align*}
and
\begin{align*}
	\|\nabla m\|_{L^2} 
	&\leq \|\nabla\rho\|_{L^3}\|u\|_{L^6}+\|\rho\|_{L^\infty}\|\nabla u\|_{L^2}\\
	&\leq C(\|\nabla(\rho-\rho_\infty)\|_{H^2}+\|\nabla\rho_\infty\|_{H^2})\|\nabla u\|_{L^2}\\
	&\leq C\|\nabla u\|_{L^2},
\end{align*}
which implies that $\|\nabla u\|_{L^2}\sim \|\nabla m\|_{L^2}$. However, the second-order derivatives are different. Indeed, we find
\begin{align*}
	&\|\nabla^2u\|_{L^2}\cr
	&\quad =\left\|\frac{\nabla^2m}{\rho}-\frac{2\nabla m\cdot\nabla\rho}{\rho^2}-\frac{m\cdot\nabla^2\rho}{\rho^2}+\frac{2m |\nabla\rho|^2}{\rho^3} \right\|_{L^2}	\nonumber\\
	&\quad \leq C\|\nabla^2m\|_{L^2}+C\|\nabla m\|_{L^6}\|\nabla\rho\|_{L^3}+C\|m\|_{L^\infty}\|\nabla^2\rho\|_{L^2}+C\| m\|_{L^\infty}\|\nabla  \rho\|_{L^6}\|\nabla\rho\|_{L^3}\nonumber\\
	&\quad \leq C\|\nabla^2m\|_{L^2}+C\|\nabla^2 m\|_{L^2}(\|\nabla(\rho-\rho_\infty)\|_{H^1}+\|\nabla\rho_\infty\|_{H^1})\nonumber\\
	&\qquad+C\|\nabla m\|_{L^2}^{\frac{1}{2}}\|\nabla^2m\|_{L^2}^{\frac{1}{2}}(\|\nabla^2(\rho-\rho_\infty)\|_{L^2}+\|\nabla^2\rho_\infty\|_{L^2})\nonumber\\
	&\qquad+C\|\nabla m\|_{L^2}^{\frac{1}{2}}\|\nabla^2m\|_{L^2}^{\frac{1}{2}}(\|\nabla^2(\rho-\rho_\infty)\|_{L^2}+\|\nabla^2\rho_\infty\|_{L^2})(\|\nabla(\rho-\rho_\infty)\|_{H^1}+\|\nabla\rho_\infty\|_{H^1})\nonumber\\
	&\quad \leq C(\|\nabla m\|_{L^2}+\|\nabla^2 m\|_{L^2})\nonumber\\
	&\quad \leq C\|\nabla m\|_{H^1}
\end{align*}
and 
\begin{align*}
	\|\nabla^2m\|_{L^2}&\leq \|\nabla^2\rho\|_{L^2}\|u\|_{L^\infty}+2\|\nabla\rho\|_{L^3}\|\nabla u\|_{L^6}+\|\rho\|_{L^\infty}\|\nabla^2u\|_{L^2}\\
	&\leq C(\|\nabla^2(\rho-\rho_\infty)\|_{L^2}+\|\nabla^2\rho_\infty\|_{L^2})	\|\nabla u\|_{H^1}\\
	&\quad+C(\|\nabla(\rho-\rho_\infty)\|_{H^1}+\|\nabla\rho_\infty\|_{H^1})\|\nabla^2 u\|_{L^2}+C\|\nabla^2u\|_{L^2}\\
	&\leq C(\|\nabla u\|_{L^2}+\|\nabla^2u\|_{L^2})\\
	&\leq C\|\nabla u\|_{H^1}.
\end{align*}
Similarly, we deduce
\begin{align*}
	\|\nabla^3u\|_{L^2}\leq C\|\nabla m\|_{H^2},  \quad \|\nabla^3 m\|_{L^2}\leq C\|\nabla u\|_{H^2}.
\end{align*}
Hence, we have the following relations:
\[
	\|\nabla u\|_{H^2}\sim  \|\nabla m\|_{H^2},\quad 	\| u\|_{H^3}\sim  \|m\|_{H^3}.
\]
Combining the above estimates with \eqref{030604} implies
\[
	\int_0^t\|u(s)\|_{H^2}^2ds\leq C\int_0^t\|m(s)\|_{H^2}^2ds\leq C\gamma^{-1} (B_0^4+B_0^2). 
	\]
In particular, we have
\begin{align}\label{030606}
	\|(\rho-\rho_\infty, m,u)\|_{H^2}^2+\gamma \int_0^t\|u(s)\|_{H^2}^2ds
	\leq C (B_0^4+B_0^2) .
\end{align}
For the estimate of $\|\nabla(\phi, u)\|_{H^2}^2$,  we notice that
\begin{align*}
	\|\nabla\phi\|_{H^2}\leq C\|\nabla(\rho-\rho_\infty)\|_{H^2}\leq C\gamma^{\frac{1}{2}}\delta.
\end{align*}
Integrating \eqref{030605} with respect to time and using \eqref{030607} and \eqref{030606} yields
\begin{align*}
	&\gamma^{-1}
	\int_0^t\|\nabla \phi(s)\|_{H^2}^2ds \\
	&\quad \leq 2\gamma^{-1}\left|\sum_{k=0}^2(\nabla^k u,\nabla^{k+1}\phi)\right| +  2\gamma^{-1}\left|\sum_{k=0}^2(\nabla^k u_0,\nabla^{k+1}\phi_0)\right|  +C\gamma\int_0^t\|u(s)\|_{H^2}^2ds \\
	&\qquad +C\gamma^{-1}\int_0^t\|\nabla u(s)\|_{H^2}^2(1+\gamma^{-2}\|\nabla \phi(s)\|_{H^2}^2)ds+C\epsilon_0\gamma^{-1}\int_0^t\|\nabla u(s)\|_{H^2}^2ds \\
	&\quad \leq C\delta^2+ C(B_0^4+B_0^2) +C\gamma^{-2} (B_0^4+B_0^2) \\
	&\qquad +C\gamma^{-3}\delta^2 \int_0^t\|\nabla\phi(s)\|_{H^2}^2ds+C\epsilon_0\gamma^{-1}\delta^2.
\end{align*}
Since $\gamma^{-1}$ and $\delta$ are small, we show that there exists a positive constant $\mathcal{X}_0$ such that 
\begin{align}\label{030601}
	\gamma^{-1}
	\int_0^t\|\nabla \phi(s)\|_{H^2}^2ds\leq \mathcal{X}_0(B_0^4+B_0^2) .
\end{align}
To close the \emph{a priori} assumption, we choose 
\[
\mathcal{R}_0 := \sqrt{2\mathcal{X}_0(B_0^4+B_0^2)} ,
\]
which implies that 
\[
	\gamma^{-1}
	\int_0^t\|\nabla \phi(s)\|_{H^2}^2ds\leq \frac{1}{2}\mathcal{R}_0^2.
\]
On the other hand, it follows from \eqref{022501} that 
\[
	\frac{d}{dt}\|\nabla(\phi, u)\|_{H^2}^2+\gamma\| \nabla u\|_{H^2}^2\leq C\gamma^{-1}\|\nabla \phi\|_{H^2}^4+C\gamma^{-1}\epsilon_0^2\|\nabla\phi\|_{H^2}^2,
\]
and applying Gr\"{o}nwall's lemma together with the bound estimate \eqref{030601}, we deduce 
\begin{align*}
	&\|\nabla(\phi,u)(t)\|_{H^2}^2 +\gamma\int_0^t\|\nabla u(s)\|_{H^2}^2\,ds\cr
	&\quad \leq e^{C\gamma^{-1}\int_0^t\|\nabla\phi(s)\|_{H^2}^2ds}\left(\|\nabla(\phi_0,u_0)\|_{H^2}^2+ C\gamma^{-1}\epsilon_0^2\int_0^t\|\nabla\phi(s)\|_{H^2}^2\,ds\right)\nonumber\\
	&\quad \leq e^{C(B_0^4+B_0^2)}(B_0^2+C\epsilon_0^2(B_0^4+B_0^2))\nonumber\\
	&\quad \leq CB_0^2e^{C(B_0^4+B_0^2)} \nonumber\\
	&\quad\leq C_0,
\end{align*}
where the constant $C_0=C(B_0)>0$ only depends on $B_0$. 

 %
 %
 %
 %
 %
 %
 %
 %
 %
 %

 \vfill

\end{document}